\documentclass{article}

\usepackage[utf8]{inputenc}
\usepackage{amsmath, amssymb}
\usepackage{amsthm}
\usepackage{todonotes}
\usepackage{framed}
\usepackage{makecell}
\usepackage{subcaption}
\usepackage{mathtools}
\usepackage{comment}
\usepackage{bm}
\usepackage{xspace}
\usepackage{multirow}
\usepackage[T1]{fontenc}
\usepackage{hyperref}
\usepackage{booktabs}
\usepackage{lmodern}
\usepackage{microtype}
\usepackage[top=1.8cm,left=2cm,right=2cm,bottom=2cm]{geometry}

\renewcommand{\k}{K}
\newcommand{\X}{\mathcal{X}}
\newcommand{\Xstatic}{\mathcal{X}^{\textsc{sta}}}
\newcommand{\Xstaticb}{\tilde{\mathcal{X}}^{\textsc{sta}}}
\newcommand{\Xtwostage}{\mathcal{X}^{\textsc{two}}}
\newcommand{\U}{\mathcal{U}}
\newcommand{\W}{\mathcal{W}}
\newcommand{\Y}{\mathcal{Y}}
\newcommand{\Xc}{\mathcal{X}_c}
\newcommand{\Uc}{\mathcal{U}_c}
\newcommand{\Ucx}{\mathcal{U}_c(x)}
\newcommand{\Yc}{\mathcal{Y}_c}
\newcommand{\Xd}{\mathcal{X}_d}
\newcommand{\Ud}{\mathcal{U}_d}
\newcommand{\Udx}{\mathcal{U}_d(x)}
\newcommand{\barUd}{\bar{\mathcal{U}}_d}
\newcommand{\Yd}{\mathcal{Y}_d}
\newcommand{\R}{\mathbb{R}}
\newcommand{\Z}{\mathbb{Z}}

\newcommand{\Q}{\mathcal{Q}}

\renewcommand{\S}{\mathcal{S}}
\newcommand{\one}{\mathbf{1}}
\newcommand{\zero}{0}

\newcommand{\mx}{{m_x}}
\newcommand{\my}{{m_y}}
\newcommand{\muu}{{m_u}}

\newcommand{\ncx}{{n_{\check x}}}
\newcommand{\ncy}{{n_{\check y}}}

\newcommand{\nxt}{{n_x}}

\newcommand{\nut}{{n_u}}

\newcommand{\nyt}{{n_y}}
\newcommand{\nzc}{{n_z^c}}
\newcommand{\nzd}{{n_z^d}}
\newcommand{\nzt}{{n_z}}

\newcommand{\PH}{\textsf{PH}}
\newcommand{\PSPACE}{\textsf{PSPACE}}
\renewcommand{\P}{\mathsf{P}}

\newcommand{\NP}{\mathsf{NP}}
\newcommand{\coNP}{\mathsf{coNP}}

\newcommand{\bilevel}{\textsc{Bilevel}\xspace}

\newcommand{\twostage}{\textsc{Two-Stage}\xspace}
\newcommand{\static}{\textsc{Static}\xspace}
\newcommand{\finiteadapt}{\textsc{Finite-Adaptability}\xspace}
\newcommand{\Kadapt}[1]{\textsc{$#1$-Adaptability}\xspace}
\newcommand{\boundedtwostage}{\textsc{Bounded-Two-Stage}\xspace}
\newcommand{\boundedstatic}{\textsc{Bounded-Static}\xspace}
\newcommand{\boundedfiniteadapt}{\textsc{Bounded-Finite-Adaptability}\xspace}
\newcommand{\boundedKadapt}[1]{\textsc{Bounded-$#1$-Adaptability}\xspace}

\newcommand{\QIP}{\textsc{QIP}\xspace}
\newcommand{\boundedQP}{\textsc{Bounded-QP}\xspace}
\newcommand{\QP}{\textsc{QP}\xspace}
\newcommand{\QMIP}{\textsc{QMIP}\xspace}
\newcommand{\boundedQMIPsingle}{\textsc{Bounded-QMIP-Single}\xspace}
\newcommand{\QMIPsingle}{\textsc{QMIP-Single}\xspace}

\newcommand{\set}[2]{\left\{#1 \; \left|\;\; #2 \right.\right\}}

\newtheorem{question}{Question}
\newtheorem{fact}{Fact}
\newtheorem{lemma}{Lemma}
\newtheorem{theorem}{Theorem}

\newtheorem{observation}{Observation}

\newtheorem{corollary}{Corollary}

\DeclareMathOperator*{\argmin}{arg\,min}

\DeclareMathOperator{\ext}{ext}
\DeclareMathOperator{\ray}{ray}
\DeclareMathOperator{\conv}{conv}

\providecommand{\enc}[1]{\langle #1 \rangle}            
\providecommand{\fc}{\varphi}

\title{The complexity landscape of robust (integer) linear programming}
\author{Michael Poss\footnote{LIRMM, University of Montpellier, CNRS, France, \texttt{michael.poss@lirmm.fr}} \and Jannis Kurtz\footnote{University of Amsterdam, the Netherlands, \texttt{j.kurtz@uva.nl}} \and Marc Goerigk\footnote{Business Decisions and Data Science, University of Passau, Germany, \texttt{marc.goerigk@uni-passau.de}} \and Dorothee Henke\footnote{Institute of Business Analytics and Technology Transformation, Johannes Kepler University Linz, Austria; Business Decisions and Data Science, University of Passau, Germany, \texttt{dorothee.henke@uni-passau.de}}}
\date{}

\begin{document}

\maketitle

\begin{abstract}
We study the computational complexity of the decision versions of three classic classes of robust optimization problems: static robust optimization, two-stage (adjustable) robust optimization, and $K$-adaptability. We consider that the feasibility sets, the uncertainty sets, and the recourse sets are polyhedra or integer-programming-representable sets, the uncertainty is decision-independent or decision-dependent, and the feasible regions are bounded or unbounded.
While these problems are well-established in the current literature, we give a systematic classification that places them in the complexity classes $\P$, $\NP$, $\coNP$, and higher levels of the polynomial hierarchy ($\Sigma_2^p$ and $\Sigma_3^p$), and, once decision-dependent uncertainty introduces quadratic constraints, in the existential theory of the reals and its hierarchy ($\exists\R$, $\Sigma_2\R$, and $\Sigma_3\R$) or among the undecidable problems. Beyond hardness reductions, we pay particular attention to membership proofs, establishing polynomial-size certificates even though the sets considered generally contain vectors of exponential encoding length. As a by-product, we give an alternative proof that bilevel linear optimization lies in $\NP$, exploiting its connection to decision-dependent robust optimization. For $K$-adaptability, we relate the problem to its static and two-stage counterparts, showing that hardness grows monotonically with $K$ and that, for fully discrete decision-dependent instances, $K$-adaptability reduces back to the static problem.\\

\end{abstract}

\section{Introduction}

In this work, we consider the computational complexity of three classic classes of robust optimization problems. Accordingly, we introduce decision versions of our problems in the following.

\subsection{Problem definitions}

The first problem class is static robust optimization, where one wishes to find a solution $x$ that remains feasible (with respect to the $L$ uncertain constraints) for any realization of an uncertain parameter $u$.
Introducing the first-stage feasibility set $\X\subseteq\R^\nxt$, the uncertainty set $\emptyset\neq\U(x)\subseteq\R^\nut$, which may depend on $x$, and the coupling matrix $H_\ell$ for each $\ell\in[L]=\{1,\ldots,L\}$, the decision problem one wishes to solve for threshold $v$ is:
\begin{equation}
    \label{eq:staticminmax}
\min_x\set{\max_u\set{u^\top H_1 x}{u\in\U(x)}}{x\in\X,\;u^\top H_\ell x\leq 0,\;\forall \ell\in[L]\setminus\{1\}, u\in\U(x)}\leq v,
\end{equation}
which is equivalent to (assuming $v=0$)
\begin{equation}
    \label{eq:staticDEC}\tag{\static}
\exists x\in \X\ \forall u\in \U(x):u^\top H_\ell x\leq 0,\;\forall \ell\in[L]?
\end{equation}
Observe that, in addition to $v=0$, \static assumes without loss of generality that the right-hand side of each inequality is 0 (since some components of~$x$ and~$u$ could be assumed to be constant).

As a natural extension to \static, we consider the two-stage robust optimization problem wherein the decision maker is able to react to the realization of the uncertain parameters.
We denote the recourse feasibility set, which may depend on $x$ and $u$, by $\Y(x,u)\subseteq\R^\nyt$.
The two-stage robust problem we study is:
\begin{equation}
    \label{eq:2roDEC}\tag{\twostage}
\exists x\in \X\ \forall u\in \U(x)\ \exists y \in\Y(x,u)?
\end{equation}
Observe that $\Y(x,u)$ encompasses the second-stage feasibility as well as the possible objective function of the related optimization problem. One readily verifies that~\eqref{eq:2roDEC} is equivalent to the decision version of min-max-min formulations of two-stage robust optimization typically used in the literature, in the line of~\eqref{eq:staticminmax}.

Last, we consider a variant of \twostage that supposes that one must decide on up to $K$ second-stage decisions $y^1,\ldots,y^K$ before the uncertain parameters are revealed, given that one of these $K$ decisions is chosen upon the full knowledge of $u$. In other words, these $K$ second-stage decision vectors $y^1,\ldots,y^K$ are chosen to ensure that, for each realization of the uncertain vector, at least one of these decision vectors is feasible for the set $\Y(x,u)$, leading to the $K$-adaptability decision problem:
\begin{equation}
    \label{eq:k-adapt}\tag{\Kadapt{\k}}
\exists x\in \X, y^1,\ldots ,y^K\in \R^{\nyt}\ \forall u\in \U(x): \{y^1,\ldots,y^\k\} \cap \Y(x,u)\neq\emptyset?
\end{equation}
In this problem, we assume that $K$ is given by a function of the input, the function itself not being part of the input; this includes constant $K$ as a special case. We furthermore define \finiteadapt as the variant of \Kadapt{\k} where $\k$ is a value which is part of the input.
As for the problem \twostage, the possible objective function is implicitly modeled by $\Y(x,u)$, and one readily verifies that~\eqref{eq:k-adapt} is equivalent to the decision version of the min-max-min formulation of~\cite{hanasusanto2015k}.

This paper focuses on two types of constraints for $x$, $u$, and $y$. On the one hand, we consider the case where the constraints form polyhedra in outer representation. In this case, we add the index $\bullet_c$ (for ``continuous'') to the sets:
\begin{align*}
\Xc&\equiv\set{x\in\R_{\ge 0}^\nxt}{A x \leq a},\\
\Ucx&\equiv\set{u\in\R_{\ge 0}^\nut}{B(x)u \leq b},\\
\Yc(x,u)&\equiv\set{y\in\R_{\ge 0}^\nyt}{T(u)x + W y \leq 0},
\end{align*}
where $A \in \mathbb{Q}^{\mx \times \nxt}$, $a\in\mathbb{Q}^{\mx}$, $b\in\mathbb{Q}^{\muu}$, and $W \in \mathbb{Q}^{\my \times \nyt}$. We further assume that $B:\X\rightarrow\mathbb{Q}^{\muu\times\nut}$ and $T:\bigcup_{x \in \X} \U(x)\rightarrow\mathbb{Q}^{\my\times\nxt}$ are linear in $x$ and $u$, respectively.
Observe that we can account for affine dependencies, instead of linear ones, by introducing auxiliary variables $x_{\nxt + 1}=1$ and $u_{\nut + 1}=1$. Similarly, it is well-known that the independence of $b$ from $x$
is made without loss of generality.
Last, we also assume without loss of generality that the right-hand side of the constraints describing $\Yc(x,u)$ is equal to 0,
similarly to the definition of \static above.

On the other hand, we consider the discrete counterparts of the above sets, denoted by the index $\bullet_d$ (for ``discrete'') and formally defined as $\Xd=\Xc\cap\Z^\nxt$, $\Udx=\Ucx\cap\Z^\nut$, and $\Yd(x, u)=\Yc(x, u) \cap\Z^\nyt$.  If $B$ is a constant, i.e., the uncertainty set is not decision-dependent, we write $\Uc$ and $\Ud$ instead of $\Uc(x)$ and $\Ud(x)$, respectively.

Instances for the above problems are obtained by describing the specific sets $\X$, $\U$, and $\Y$, and their possible dependencies on $x$ and $u$, with the necessary vectors and matrices with rational entries. Our study considers different variants of the problems, depending on whether $\X$, $\U$, and $\Y$ are continuous or discrete, and whether or not $\U$ depends on $x$. When we wish to focus on the problem considering only a subset of instances satisfying the above restrictions, we use the notation $\static(\X,\U)$, $\twostage(\X,\U,\Y)$, $\finiteadapt(\X,\U,\Y)$, and $\Kadapt{K}(\X,\U,\Y)$ for the cases $\X\in\{\Xc,\Xd\}$, $\U\in\{\Uc,\Ud,\Ucx,\Udx\}$, and $\Y\in\{\Yc,\Yd\}$. Note that, for ease of notation, we write $\Yc$ and $\Yd$ instead of $\Yc(x, u)$ and $\Yd(x, u)$ here.

The above problems do not assume that the underlying sets are bounded. As we will see throughout the paper, this leads to undecidability results for several variants for which the uncertainty set is decision-dependent. While studying the complexity of the unbounded case is of theoretical interest, practical problems typically consider bounded sets, so we additionally define the bounded versions of the above problems. These versions assume that each of the considered sets, $\X$, $\U$, and $\Y$, is contained in an $\ell_\infty$-ball of radius $R$, which is part of the problem input. For instance, in case of $\boundedstatic$, we define $\Xc\equiv\set{x\in[0,R]^\nxt}{A x \leq a}$. Similarly, we extend the definitions of $\U$ and $\Y$, leading to the bounded problems $\boundedtwostage$, $\boundedfiniteadapt$ and $\boundedKadapt{K}$. Observe that, for each of these problems, the bounded variant is a special case of the original problem. Therefore, the hardness of the bounded variant implies the hardness of the unbounded problem, while proving the membership of the unbounded problem in some complexity class implies the membership of the bounded variant in the same complexity class.

\subsection{Related literature}

Many authors have studied the complexity of robust optimization problems, starting with \static for $\X=\Xd$ and the uncertainty set being described as an explicit list of scenarios; see, for instance,~\cite{kouvelis2013robust,robook}. Robust optimization with convex uncertainty sets has emerged in parallel to this line of research on discrete scenarios~\cite{Ben-TalN98}. The results from~\cite{Ben-TalN98,kouvelis2013robust} immediately yield the polynomial-time solvability of $\static(\Xc,\Uc)$ and the $\NP$-completeness of $\static(\Xd,\Uc)$, though special polyhedra are known to lead to problems where \static is not harder than their nominal counterparts~\cite{bertsimas2003robust},\cite[Corollary 2]{OJMO_2024__5__A5_0}.

Subsequent authors have considered the complexity of \twostage, again mostly for $\X=\Xd$ and $\Y=\Yd$~\cite{robook}.
Such problems with alternating quantifiers are a special case of quantified integer programming, which is known to be $\Sigma_3^p$-hard in general, while special cases such as a constant number of variables may reduce the complexity~\cite{nguyen2020computational}.
Possibly one of the most striking results in this line of research lies in the meta-theorems proposed by~\cite{GruneW25,GruneW25b}, which transfer $\NP$-completeness from the nominal problem to $\Sigma_3^p$-hardness of the \twostage variant, under some technical assumptions. Yet \twostage is known to be difficult, even when there are no discrete variables at all. Specifically,~\cite{guslitser2002uncertainty} shows that \twostage is $\NP$-hard even when $\U=\set{u\in[-1,1]^n}{\sum_{i = 1}^n a_i u_i =0}$, using a reduction from the partition problem. Meanwhile,~\cite{chekuri2007hardness} shows that the separation problem is $\coNP$-hard in the case where $\U$ is the Hose model.

\Kadapt{K}, initially introduced as a practical and hopefully more efficient way to handle discrete second-stage variables~\cite{bertsimas2010finite,HanasusantoKW16}, has since then been the topic of diverse theoretical studies, in particular, regarding approximation guarantees one can expect by considering small values of $\k$~\cite{Kurtz24} as well as studying which such value one should consider to reach optimality~\cite{kurtz2026bounding}. In their seminal paper~\cite{bertsimas2010finite}, Bertsimas and Caramanis prove that \Kadapt{2} is $\NP$-hard. In \cite{HanasusantoKW16}, the authors show that evaluating the objective function of the binary robust $K$-adaptability problem with uncertainty in the constraints can be done in polynomial time if $\k$ is fixed, but is strongly $\NP$-hard otherwise. A related problem called \textit{min-max-min robust optimization} \cite{buchheim2017min}, where no first-stage decisions are considered, was studied in terms of complexity mostly for the case of objective uncertainty. In \cite{buchheim2017min}, the authors reduce the $\k$-adaptability problem to the $(\k+1)$-adaptability problem for binary problems with polyhedral uncertainty. Additionally, they show that, if the underlying problem is polynomial-time solvable, then the min-max-min robust problem is polynomial-time solvable if $\k$ is at least the dimension of the problem. Further complexity results were derived for several problems with discrete uncertainty sets \cite{buchheim2018complexity,goerigk2020min}, for $\k=2$ \cite{chassein2021complexity} and for a treatment planning problem in radiotherapy \cite{qiu2025k}.

The above references all consider decision-independent uncertainty sets. However, richer robust models can account for decision-dependence motivated, for instance, by probabilistic guarantees~\cite{Poss13} or applications in which the uncertainty can be reduced by spending resources to acquire
information~\cite{arslan2024uncertainty,nohadani2018optimization}. Decision-dependence is also central to the more recent model of decision-dependent information discovery, introduced by~\cite{vayanos2026robust} to allow decision makers to query the most crucial sources of uncertainty. Last, it allows to establish a relationship between robust optimization and bilevel optimization~\cite{OJMO_2025__6__A2_0}.

\subsection{Our contributions}

Despite the rich history of complexity studies in robust optimization, some of which are summarized above, no work so far has carried out a systematic study of the complexity of \static, \twostage, \finiteadapt, and \Kadapt{\k} driven by classifying the sets $\X$, $\U$, and $\Y$ as we do herein: polytopes, polyhedra, and integer-programming-representable sets (i.e., sets containing all integer points within some polyhedron), bounded or not. As a consequence, our results shed new light on the complexity of robust (integer) linear programming. Our results are of two types. On the one hand, we provide many-one reductions to prove hardness of our problems for a given complexity class. On the other hand, and perhaps less common in the mathematical optimization literature, we spend significant effort to prove membership in these complexity classes. Indeed, all sets considered herein, apart from bounded integer-programming-representable ones, contain vectors that have no polynomial encoding length, so we must prove the existence of small-size certificates. While the membership can be obtained quite naturally in some settings by using folklore results of integer programming, other results require more insights, such as \twostage with discrete second-stage variables or decision-dependent problems, and we have not been able to prove membership for \twostage with discrete (unbounded) second-stage variables. Our results, detailed in Tables~\ref{tab:static},~\ref{tab:twostage}, and~\ref{tab:k-adapt}, and further illustrated in Figure~\ref{fig:complexity}, end up classifying many variants of the problems \static, \twostage, and \Kadapt{\k} not only within the polynomial hierarchy, but also within the theory of the reals and the class of undecidable problems. The latter connections naturally appear when considering decision-dependent uncertainty sets, which lead to quadratic constraints, which are, in turn, known to lead to decision problems that are either complete for the existential theory of the reals, or undecidable, depending on the presence of integer variables or not. These results thus provide theoretical justification for the hardness of decision-dependent robust optimization observed in practice. Additionally, we provide an alternative proof of the membership of bilevel linear programming in $\NP$, which was established in~\cite{Buchheim23}. Here, we leverage the connection of bilevel linear programming to decision-dependent robust optimization already illustrated in~\cite{OJMO_2025__6__A2_0}.

For $\Kadapt{\k}$, rather than investigating all cases as for \static and \twostage, we relate the problem to \static and \twostage through many-one reductions. The main results that we are able to prove in this area are summarized below. First, while observing that $\boundedstatic(\X,\U)$ trivially reduces to $\boundedKadapt{1}(\X,\U,\Y)$, we provide a reduction from $\boundedKadapt{1}(\X,\U,\Y)$ to $\boundedKadapt{K}(\X,\U,\Y)$, showing that hardness grows monotonically with $K$.
\begin{theorem}\label{thm:k_k+1_reduction_intro}
Let $\X\in \{ \X_c,\X_d\}$, $\U\in \{\U_c, \U_d\}$, and $\Y\in \{\Y_c,\Y_d\}$. If $\k$ is polynomial in the input size, then $\boundedKadapt{1}(\X,\U,\Y) \le \boundedKadapt{K}(\X,\U,\Y).$ 
\end{theorem}
Additionally, we can prove that, in slightly more restrictive settings, $\boundedKadapt{K}$ can be reduced to $\boundedtwostage$. This places $\boundedKadapt{K}$ between $\boundedstatic$ and $\boundedtwostage$ in terms of hardness. Through additional membership proofs, we show that, in some cases, $\boundedKadapt{K}$ can be located on a higher level in the polynomial hierarchy than $\boundedstatic$.

Perhaps more surprising is the fact that decision-dependence allows us to prove the reverse reduction in some case.
\begin{theorem}
If $\k$ is polynomial in the input size, then $\boundedKadapt{\k}(\Xd,\Udx,\Yd) \le\newline \boundedKadapt{1}(\Xd,\Udx,\Yd).$
\end{theorem}

The rest of the paper is structured as follows. In the next section, we introduce additional notation, recall the basic concepts and definitions about the complexity classes used throughout, and provide folklore results about the membership of solutions to mixed-integer linear and quadratic optimization problems. Sections~\ref{sec:static},~\ref{sec:twostage}, and~\ref{sec:kadapt} dive into the results summarized in Tables~\ref{tab:static},~\ref{tab:twostage}, and~\ref{tab:k-adapt} for \static, \twostage, and \Kadapt{K}, respectively.

\begin{figure}
    \centering
    \scalebox{0.8}{
\begin{tikzpicture}

\definecolor{cUndec}{RGB}{190,190,190}
\definecolor{cHard} {RGB}{245,245,245}
\definecolor{cS3p}  {RGB}{255,185,105}
\definecolor{cS2p}  {RGB}{255,252,130}
\definecolor{cCoNP} {RGB}{135,205,255}
\definecolor{cNP}   {RGB}{135,240,165}
\definecolor{cP}    {RGB}{230,255,235}
\definecolor{oS3p}  {RGB}{200,110, 20}
\definecolor{oS2p}  {RGB}{180,160,  0}
\definecolor{oCoNP} {RGB}{ 20,120,200}
\definecolor{oNP}   {RGB}{ 20,160, 60}
\definecolor{oUndec}{RGB}{100,100,100}
\definecolor{oExR}  {RGB}{130, 30,200}
\definecolor{oS2R}  {RGB}{210, 20,130}
\definecolor{oS3R}  {RGB}{210, 30, 30}

%
%

\fill[cUndec, rounded corners=8pt] (16.6,8.05) rectangle (20.8,12.3);
\fill[cHard,  rounded corners=8pt] (16.6,2.55) rectangle (20.8, 7.7);
\fill[cS3p,   rounded corners=8pt] (0.2, 2.7)  rectangle (9.8, 12.0);
\fill[cS2p,   rounded corners=7pt] (1.0, 3.0)  rectangle (8.8,  9.8);
\fill[cCoNP,  rounded corners=6pt] (1.6, 3.3)  rectangle (5.8,  6.0);
\fill[cNP,    rounded corners=6pt] (4.0, 3.3)  rectangle (8.2,  6.0);
\fill[cP,     rounded corners=4pt] (4.2, 3.6)  rectangle (5.6,  5.0);

\draw[oUndec, line width=1.0pt, rounded corners=8pt]
     (16.6,8.05) rectangle (20.8,12.3);
\draw[oUndec, densely dotted, line width=1.0pt, rounded corners=8pt]
     (16.6,2.55) rectangle (20.8, 7.7);
\draw[oS3p,   line width=1.4pt, rounded corners=8pt]
     (0.2, 2.7)  rectangle (9.8, 12.0);
\draw[oS2p,   line width=1.0pt, rounded corners=7pt]
     (1.0, 3.0)  rectangle (8.8,  9.8);
\draw[oCoNP,  line width=1.0pt, rounded corners=6pt]
     (1.6, 3.3)  rectangle (5.8,  6.0);
\draw[oNP,    line width=1.0pt, rounded corners=6pt]
     (4.0, 3.3)  rectangle (8.2,  6.0);
\draw[black,  line width=0.7pt, rounded corners=4pt]
     (4.2, 3.6)  rectangle (5.6,  5.0);

\draw[oS3R, dash dot dot, line width=1.1pt, rounded corners=7pt]
     (-0.1, 2.55) rectangle (16.0,12.3);
\draw[oS2R, dash dot,     line width=1.1pt, rounded corners=7pt]
     (0.7,  2.85) rectangle (14.5,10.1);
\draw[oExR, dashed,       line width=0.8pt, rounded corners=6pt]
     (3.5,  3.15) rectangle (12.5, 6.3);

\node[font=\normalsize\bfseries, black]                   at (4.9, 4.6)  {$\mathsf{P}$};
\node[font=\normalsize\bfseries, oNP,  anchor=north east] at (8.1, 5.9)  {$\mathsf{NP}$};
\node[font=\normalsize\bfseries, oCoNP,anchor=north west] at (1.7, 5.9)  {$\mathsf{coNP}$};
\node[font=\normalsize\bfseries, oS2p, anchor=north east] at (8.7, 9.7)  {$\Sigma_2^p$};
\node[font=\normalsize\bfseries, oS3p, anchor=north east] at (9.7,11.9)  {$\Sigma_3^p$};
\node[font=\normalsize\bfseries, oExR]                    at (11.5, 5.95) {$\exists\mathbb{R}$};
\node[font=\normalsize\bfseries, oS2R] at (13.7, 4.5) {$\Sigma_2\mathbb{R}$};
\node[font=\normalsize\bfseries, oS3R] at (15.2, 6.0) {$\Sigma_3\mathbb{R}$};
\node[font=\normalsize\bfseries, oUndec, anchor=north]    at (18.7,12.2){Undecidable};
\node[font=\normalsize\bfseries, oUndec, anchor=north]    at (18.7, 7.6){Hardness only};


\node[font=\scriptsize] at (4.9,3.95) {$\mathrm{S}(\Xc,\Uc)$};

\node[font=\scriptsize, align=center] at (7.0,4.45) {
  $\mathrm{S}(\Xd,\Uc)$\\[1pt]$\mathrm{S}^b(\Xd,\Ucx)$\\[1pt]
  $\mathrm{K}(\Xd,\Uc,\Yc)$\\[1pt]$\mathrm{K}(\Xd,\Uc,\Yd)$\\[1pt]
  $\mathrm{K}(\Xd,\Ucx,\Yc)$\\[1pt]$\mathrm{K}(\Xd,\Ucx,\Yd)$
};

\node[font=\scriptsize, align=center] at (3.0,4.6) {
  $\mathrm{S}(\Xc,\Ud)$\\[2pt]
  $\mathrm{T}(\Xc,\Uc,\Yc)$\\[2pt]
  $\mathrm{T}(\Xc,\Ud,\Yc)$
};

\node[font=\scriptsize] at (11.15,4.5) {$\mathrm{S}(\Xc,\Ucx)$};

\node[font=\scriptsize, align=center] at (5.0,7.95) {
  $\mathrm{S}(\Xd,\Ud)$\quad $\mathrm{S}^b(\Xc,\Udx)$\quad $\mathrm{S}^b(\Xd,\Udx)$\\[3pt]
  $\mathrm{T}(\Xd,\Uc,\Yc)$\quad $\mathrm{T}(\Xd,\Ud,\Yc)$\\[3pt]
  $\mathrm{T}^b(\Xc,\Udx,\Yc)$\quad $\mathrm{T}^b(\Xd,\Ucx,\Yc)$\quad $\mathrm{T}^b(\Xd,\Udx,\Yc)$\\[6pt]
  $\mathrm{K}(\Xd,\Ud,\Yc)$\quad $\mathrm{K}(\Xc,\Ud,\Yd)$\quad $\mathrm{K}(\Xd,\Ud,\Yd)$\\[3pt]
  $\mathrm{K}(\Xc,\Udx,\Yc)$\quad $\mathrm{K}(\Xd,\Udx,\Yc)$\quad $[\mathrm{K}(\Xc,\Ud,\Yc)]$\\[3pt]
  $\mathrm{K}(\Xc,\Udx,\Yd)$\quad $\mathrm{K}(\Xd,\Udx,\Yd)$
};

\node[font=\scriptsize, align=center] at (5.0,11.05) {
  $\mathrm{T}^b(\Xd,\Uc,\Yd)$\quad $\mathrm{T}^b(\Xc,\Ud,\Yd)$\quad $\mathrm{T}^b(\Xd,\Ud,\Yd)$\\[3pt]
  $\mathrm{T}^b(\Xd,\Ucx,\Yd)$\quad $\mathrm{T}^b(\Xc,\Udx,\Yd)$\quad $\mathrm{T}^b(\Xd,\Udx,\Yd)$
};

\node[font=\scriptsize] at (12.0,9.4) {$[\mathrm{T}(\Xc,\Ucx,\Yc)]$};

\node[font=\scriptsize] at (12.0,11.05) {$[\mathrm{T}^b(\Xc,\Ucx,\Yd)]$};

\node[font=\scriptsize, align=center] at (18.7,9.95) {
  $\mathrm{S}(\Xd,\Ucx)$\\[3pt]
  $\mathrm{S}(\Xd,\Udx)$\\[3pt]
  $\mathrm{T}(\Xd,\Ucx,\Yc)$\\[3pt]
  $\mathrm{T}(\Xd,\Udx,\Yc)$\\[3pt]
  $\mathrm{T}(\Xc,\Ucx,\Yd)$\\[3pt]
  $\mathrm{T}(\Xd,\Ucx,\Yd)$\\[3pt]
  $\mathrm{T}(\Xc,\Udx,\Yd)$\\[3pt]
  $\mathrm{T}(\Xd,\Udx,\Yd)$
};

\node[font=\scriptsize, align=center] at (18.7,4.95) {
  {\color{oNP}\bfseries $\mathsf{NP}$-hard}\\[3pt]
  $\mathrm{K}(\Xc,\Uc,\Yc)$, $K\geq 2$\\[3pt]
  $\mathrm{K}(\Xc,\Uc,\Yd)$\\[7pt]
  {\color{oExR}\bfseries $\exists\mathbb{R}$-hard}\\[3pt]
  $\mathrm{K}(\Xc,\Ucx,\Yc)$\\[3pt]
  $\mathrm{K}(\Xc,\Ucx,\Yd)$\\[7pt]
  {\color{oS3p}\bfseries $\Sigma_3^p$-hard}\\[3pt]
  $\mathrm{T}^b(\Xc,\Uc,\Yd)$
};

\end{tikzpicture}
    }
\caption{Illustration of the results from Tables \ref{tab:static}--\ref{tab:k-adapt}. $\mathrm{S}$ = \textsc{Static} (Table~\ref{tab:static}), $\mathrm{T}$ = \textsc{Two-Stage} (Table~\ref{tab:twostage}), $\mathrm{K}$ = \textsc{Bounded-$K$-Adaptability} with $K$ fixed (Table~\ref{tab:k-adapt}), ${}^b$ = bounded only (unbounded: undecidable or membership open), $[\cdot]$ = class membership proven, completeness unknown.
\label{fig:complexity}}

\end{figure}

\begin{table}
    \centering
    \begin{tabular}{llll}
      \toprule
    & & $\Xc$ & $\Xd$ \\
      \midrule
    \multirow{2}{*}{$\Uc$} & \boundedstatic & $\P$ & $\NP$-complete \\
    & \static & $\P$ & $\NP$-complete \\
      \midrule
    \multirow{2}{*}{$\Ud$} & \boundedstatic &  $\coNP$-complete (Thm.~\ref{thm:static_coNPC}) &
    $\Sigma_2^p$-complete (\cite{claus2020note},~Thm.~\ref{thm:XcXdSigma}) \\
    & \static & $\coNP$-complete (Thm.~\ref{thm:static_coNPC}) &
    $\Sigma_2^p$-complete (\cite{claus2020note},~Thm.~\ref{thm:XcXdSigma}) \\
      \midrule
    \multirow{2}{*}{$\Ucx$} & \boundedstatic & $\exists\mathbb{R}$-complete (Cor.~\ref{cor:XcUcx}) & $\NP$-complete (Thm.~\ref{thm:boundedXdUcx}) \\
    & \static & $\exists\mathbb{R}$-complete (Cor.~\ref{cor:XcUcx}) &  undecidable (Cor.~\ref{cor:XdUcx}) \\
      \midrule
    \multirow{2}{*}{$\Udx$} & \boundedstatic & $\Sigma_2^p$-complete (Thm.~\ref{thm:XcUdx}) & $\Sigma_2^p$-complete (\cite{claus2020note}, Obs.~\ref{obs:boundedSigma2p}) \\
    & \static & $\Sigma_2^p$-hard (Thm.~\ref{thm:XcUdx}) & undecidable (Cor.~\ref{cor:XUdx}) \\
      \bottomrule
    \end{tabular}
    \caption{Complexity of \boundedstatic and \static \label{tab:static}}
  \end{table}

\begin{table}
\centering
\begin{tabular}{lllll}
  \toprule
 & & & $\Xc$ & $\Xd$ \\
\midrule
\multirow{10}{*}{$\Yc$} &
\multirow{2}{*}{$\Uc$}
& \boundedtwostage & $\coNP$-complete (Thm.~\ref{thm:2stagescoNPC}) & $\Sigma_2^p$-complete (\cite{lefebvre2025exact}, Thm.~\ref{thm:inSigma2p}) \\
& & \twostage & $\coNP$-complete (Thm.~\ref{thm:2stagescoNPC}) & $\Sigma_2^p$-complete (\cite{lefebvre2025exact}, Thm.~\ref{thm:inSigma2p}) \\
\cmidrule{2-5}
& \multirow{2}{*}{$\Ud$}
& \boundedtwostage & $\coNP$-complete (Thm.~\ref{thm:2stagescoNPC}) & $\Sigma_2^p$-complete (Obs.~\ref{obs:XdUd2stages}, Thm.~\ref{thm:inSigma2p}) \\
& & \twostage & $\coNP$-complete (Thm.~\ref{thm:2stagescoNPC}) & $\Sigma_2^p$-complete (Obs.~\ref{obs:XdUd2stages}, Thm.~\ref{thm:inSigma2p}) \\
\cmidrule{2-5}
& \multirow{4}{*}{$\Ucx$}
& \multirow{2}{*}{\boundedtwostage} & $\exists\R$-hard (Cor.~\ref{cor:XcUcx}, Obs.~\ref{obs:staticleqtwostage}),
& \multirow{2}{*}{$\Sigma_2^p$-complete (\cite{lefebvre2025exact}, Obs.~\ref{obs:XdUxYc})} \\
&&& $\Sigma_2^p$-hard (Thm.~\ref{thm:XcUcxYcH}), $\in \Sigma_2\R$ (Obs.~\ref{obs:XcUcxYc}) \\
& & \multirow{2}{*}{\twostage} & $\exists\R$-hard (Cor.~\ref{cor:XcUcx}, Obs.~\ref{obs:staticleqtwostage}),
& \multirow{2}{*}{undecidable (Obs.~\ref{obs:undecidableeasy})} \\
&&& $\Sigma_2^p$-hard (Thm.~\ref{thm:XcUcxYcH}), $\in \Sigma_2\R$ (Obs.~\ref{obs:XcUcxYc}) \\
\cmidrule{2-5}
& \multirow{2}{*}{$\Udx$}
& \boundedtwostage & $\Sigma_2^p$-complete (Thm.~\ref{thm:XcUdx}, \ref{thm:XcUdYcDD})
& $\Sigma_2^p$-complete (Obs.~\ref{obs:XdUd2stages}, \ref{obs:XdUxYc}) \\
& & \twostage & $\Sigma_2^p$-hard (Thm.~\ref{thm:XcUdx})
& undecidable (Obs.~\ref{obs:undecidableeasy}) \\
\midrule
\multirow{9}{*}{$\Yd$}
& \multirow{2}{*}{$\Uc$}
& \boundedtwostage & $\Sigma_3^p$-hard (Obs.~\ref{obs:calld_hard})
& $\Sigma_3^p$-complete (Thm.~\ref{thm:XdUcYdinSigma}, \cite{GoerigkLW24})  \\
& & \twostage & $\Sigma_3^p$-hard (Obs.~\ref{obs:calld_hard})
& $\Sigma_3^p$-hard (\cite{GoerigkLW24})  \\
\cmidrule{2-5}
& \multirow{2}{*}{$\Ud$}
& \boundedtwostage & $\Sigma_3^p$-complete (Obs.~\ref{obs:calld_hard}, Thm.~\ref{thm:XcUdYdinSigma})
& $\Sigma_3^p$-complete (Obs.~\ref{obs:XdUdYdinSigma}, \cite{GoerigkLW24}) \\
& & \twostage & $\Sigma_3^p$-hard (Obs.~\ref{obs:calld_hard})
& $\Sigma_3^p$-hard (\cite{GoerigkLW24}) \\
\cmidrule{2-5}
& \multirow{3}{*}{$\Ucx$}
   & \multirow{2}{*}{\boundedtwostage} & $\exists\R$-hard (Cor.~\ref{cor:XcUcx}, Obs.~\ref{obs:staticleqtwostage}),
                       & \multirow{2}{*}{$\Sigma_3^p$-complete (Thm.~\ref{thm:XdUcYdinSigma}, \cite{GoerigkLW24})} \\
                         &&& $\Sigma_3^p$-hard (Obs.~\ref{obs:calld_hard}), $\in \Sigma_3\R$ (Obs.~\ref{obs:XcUcxYdinS3R}) \\
& & \twostage & undecidable (Thm.~\ref{thm:XcUxYdundecidable})
& undecidable (Obs.~\ref{obs:undecidableeasy}) \\
\cmidrule{2-5}
& \multirow{2}{*}{$\Udx$}
& \boundedtwostage & $\Sigma_3^p$-complete (Obs.~\ref{obs:calld_hard}, Thm.~\ref{thm:XcUdYdxinSigma})
& $\Sigma_3^p$-complete (Obs.~\ref{obs:XdUdYdinSigma}, \cite{GoerigkLW24}) \\
& & \twostage& undecidable (Thm.~\ref{thm:XcUxYdundecidable})
             & undecidable (Obs.~\ref{obs:undecidableeasy}) \\
\bottomrule
\end{tabular}
\caption{Complexity of \boundedtwostage and \twostage \label{tab:twostage}}
\end{table}

\begin{table}
\centering
\begin{tabular}{llll}
  \toprule
 & & $\Xc$ & $\Xd$ \\
\midrule
\multirow{4}{*}{$\Yc$}
& $\Uc$ & $\NP$-hard for $K\geq 2$ (Thm.~\ref{thm:k_k+1_reduction}, \cite{bertsimas2010finite})
& $\NP$-complete (Thm.~\ref{thm:k_k+1_reduction}, \ref{thm:kadapt_inNP}, \cite{kouvelis2013robust}) \\
\cmidrule{2-4}
& $\Ud$ & $\coNP$-hard (Thm.~\ref{thm:static_coNPC}, \ref{thm:k_k+1_reduction}), $\in \Sigma_2^p$ (Thm.~\ref{thm:k-adapt_in_sigma2p})
& $\Sigma_2^p$-complete (Thm.~\ref{thm:k_k+1_reduction}, \ref{thm:k-adapt_in_sigma2p}, \cite{claus2020note}) \\
\cmidrule{2-4}
& $\Ucx$ & $\exists\mathbb{R}$-hard (Thm.~\ref{thm:k_k+1_reduction}, Cor.~\ref{cor:XcUcx})
& $\NP$-complete (Thm.~\ref{thm:k_k+1_reduction}, \ref{thm:kadapt_inNP}, \cite{kouvelis2013robust}) \\
\cmidrule{2-4}
& $\Udx$ & $\Sigma_2^p$-complete (Thm.~\ref{thm:k_k+1_reduction}, \ref{thm:k-adapt_in_sigma2p}, \ref{thm:XcUdx})
& $\Sigma_2^p$-complete (Thm.~\ref{thm:k_k+1_reduction}, \ref{thm:k-adapt_in_sigma2p}, \cite{claus2020note}) \\
\midrule
\multirow{4}{*}{$\Yd$}
& $\Uc$ & $\NP$-hard (Thm.~\ref{thm:k_k+1_reduction}, \cite{kouvelis2013robust})
& $\NP$-complete (Thm.~\ref{thm:k_k+1_reduction}, \ref{thm:kadapt_inNP}, \cite{kouvelis2013robust}) \\
\cmidrule{2-4}
& $\Ud$ & $\Sigma_2^p$-complete (Thm.~\ref{thm:k_k+1_reduction}, \ref{thm:k-adapt_in_sigma2p}, \cite{claus2020note})
& $\Sigma_2^p$-complete (Thm.~\ref{thm:k_k+1_reduction}, \ref{thm:k-adapt_in_sigma2p}, \cite{claus2020note}) \\
\cmidrule{2-4}
& $\Ucx$ & $\exists\mathbb{R}$-hard (Thm.~\ref{thm:k_k+1_reduction}, Cor.~\ref{cor:XcUcx})
& $\NP$-complete (Thm.~\ref{thm:k_k+1_reduction}, \ref{thm:kadapt_inNP}, \cite{kouvelis2013robust}) \\
\cmidrule{2-4}
& $\Udx$ & $\Sigma_2^p$-complete (Thm.~\ref{thm:k_k+1_reduction}, \ref{thm:k-adapt_in_sigma2p}, \cite{claus2020note})
           & $\Sigma_2^p$-complete (Thm.~\ref{thm:k_k+1_reduction}, \ref{thm:k-adapt_in_sigma2p}, \cite{claus2020note}) \\
  \bottomrule
\end{tabular}
\caption{Complexity of \boundedKadapt{\k} \label{tab:k-adapt} for fixed $K$.}
\end{table}

\section{Preliminaries}
\label{sec:preliminaries}

\paragraph{Notation}
Let $\one$ be the vector of all ones, the dimension of which is set by the context; the vectors of all zeros is noted similarly as 0, and already used in the introduction.
We use $[n]=\{1,\ldots,n\}$ to denote index sets, and write $[n]_0=\{0,1,\ldots,n\}$ if 0 is included as well.
For a rational vector or matrix $v$, we denote by $\enc{v}$ its \emph{encoding length}, i.e., the number of bits needed to encode it in binary; for a list of
objects, we write $\enc{\cdot,\ldots,\cdot}$ for the total encoding length.
The \emph{facet complexity} of a rational polyhedron in $\R^n$ is the smallest $\fc\geq n$
such that the polyhedron can be described by linear inequalities each of
encoding length at most $\fc$. Throughout, ``polynomial encoding length''
abbreviates ``encoding length polynomially bounded in the encoding length of
the instance''. For a polyhedron $P\subseteq\R^n$, we let $\ext(P)$ and $\ray(P)$ denote the
sets of its extreme points and extreme rays, respectively.

\paragraph{Polynomial hierarchy}
Several of the decision problems studied in this paper naturally involve
alternating existential and universal quantifiers. Their complexity is
therefore best described using the polynomial hierarchy~($\PH$), which
generalizes the classes $\NP$ and $\coNP$; see, e.g., \cite{wrathall1976} and \cite[Chapter~7.2]{gareyjohnson1979}.
Recall that $\NP$ contains all
decision problems for which YES-instances admit certificates of polynomial
encoding length that can be verified in polynomial time, while $\coNP$ contains
the complements of problems in $\NP$. More generally, the polynomial hierarchy
is obtained by allowing polynomial-time verification procedures with
alternating existential and universal quantifiers over polynomial-size
certificates. Its levels are denoted by $\Sigma_k^p$ and $\Pi_k^p$, depending on
whether the first quantifier is existential or universal.
In particular, $\Sigma_0^p = \Pi_0^p = \P$, $\Sigma_1^p = \NP$, and $\Pi_1^p = \coNP$.

In this work, we will mainly encounter the classes $\Sigma_2^p$ and
$\Sigma_3^p$. A problem belongs to $\Sigma_2^p$ if it can be expressed in the
form
\[
\exists X\ \forall Y : P(I,X,Y),
\]
where $P$ is decidable in polynomial time, $I$ denotes the instance, and all quantified variables have
polynomial encoding length.
Similarly, a problem belongs to $\Sigma_3^p$ if it
can be expressed as
\[
\exists X\ \forall Y\ \exists Z : P(I,X,Y,Z).
\]
Such quantifier alternations arise naturally in robust and bilevel
optimization, where one may seek a solution that remains feasible under all
realizations of uncertainty while allowing corrective or recourse decisions.
The complexity results established in this paper will therefore be stated in terms of
these levels of the polynomial hierarchy.

\paragraph{Small certificates and encoding lengths}
Apart from the bounded integer-programming-representable sets, all sets
considered in this paper contain vectors of arbitrarily large encoding length.
Consequently, placing a problem in a level of the polynomial hierarchy requires
showing that, in case of a YES-instance, the relevant quantified objects
can be chosen of polynomial encoding length. We collect here folklore results that we invoke repeatedly to this end. The first one essentially states that \textsc{Mixed-Integer Programming} is in $\NP$~(see, e.g, \cite[Section~4.8.2]{conforti2014integer} and \cite[Theorem~10.2]{schrijver1998theory}), while the second one follows from applying~\cite[Section~17.1]{schrijver1998theory} to the projection of the mixed-integer hull $P_{MI}\equiv \conv(P\cap(\R^{n^c}\times\Z^{n^d}))$ onto its integer coordinates (vertices) and \cite[Section~10.2]{schrijver1998theory} to the recession cone of $P_{MI}$.

\begin{fact}
\label{fact:milp}
Let $P\subseteq\R^{n^c+n^d}$ be a rational polyhedron of facet complexity $\fc$. If
$P\cap(\R^{n^c}\times\Z^{n^d})\neq\emptyset$, then $P\cap(\R^{n^c}\times\Z^{n^d})$ contains a vector of encoding length polynomially bounded in $\fc$.
\end{fact}

\begin{fact}
\label{fact:inthull}
Let $P\subseteq\R^{n^c+n^d}$ be a rational polyhedron of facet complexity $\fc$. Then, each element of $\ext(P_{MI})$ and $\ray(P_{MI})$ has encoding length polynomially bounded in $\fc$.
\end{fact}
The proofs of membership in $\Sigma_2^p$ presented in this paper (see, e.g., Theorems~\ref{thm:XcXdSigma}, \ref{thm:XcUdx}, and~\ref{thm:k-adapt_in_sigma2p}) all follow a common pattern. They reduce a membership claim to ``small witness'' conditions for each quantifier block plus the polynomial-time decidability of the predicate. For $\static$ (resp. $\Kadapt{\k}$), given a YES-instance $I$, this amounts to finding an outermost certificate $\hat{x}$ (resp. $\hat{x},\hat{y}^1, \ldots, \hat{y}^\k$) with polynomial encoding length. Then, given such a certificate $\hat{x}$ (resp. $\hat{x},\hat{y}^1, \ldots, \hat{y}^\k$) with polynomial encoding length, we need to show that the remaining problem, the instance of which is
$(I,\hat x)$,
is in $\coNP$, or equivalently, that its complement is in $\NP$. Hence, if
$(I,\hat x)$
is a NO-instance, there must exist a certificate $\hat u$ of polynomial encoding length, and $\ell\in[L]$, such that $\hat u^\top H_\ell \hat x > 0$, and similarly for $\Kadapt{\k}$. For showing that a \twostage problem is contained in $\Sigma_3^p$ (see, e.g., Theorems~\ref{thm:XcUdYdinSigma} and~\ref{thm:XdUcYdinSigma}), given an outermost certificate of polynomially bounded encoding length $\hat x$, we need to show that the remaining problem is in $\Pi_2^p$, or that its complement is in $\Sigma_2^p$, by constructing a certificate $\hat u$ of polynomially bounded encoding length for each NO-instance. Once the existence of $\hat x$ and $\hat u$ has been certified, the remaining problem, seeking $y\in\Y(\hat x,\hat u)$, has to be shown to be in $\NP$, e.g., using Fact~\ref{fact:milp}.

\paragraph{Existential theory of the reals}
The complexity class $\exists\R$ contains all decision problems that can be
polynomially reduced to the canonical problem
\textsc{Existential-Theory-of-the-Reals} (ETR). The latter consists of deciding
the truth of sentences of the form
\[
\exists Y_1,\ldots,\exists Y_n : F(Y_1,\ldots,Y_n),
\]
where the variables $Y_i$ take values in $\R$ and where $F$ is a quantifier-free
formula built from polynomial equalities and inequalities with integer
coefficients by means of the Boolean connectives. In this sense, ETR plays a role for
real algebraic decision problems that is analogous to the role of
\textsc{Satisfiability} for $\NP$.

More generally, the existential theory of the reals extends to a hierarchy with
alternating quantifiers over real variables, analogous to the polynomial
hierarchy~\cite{schaefer2024existential}. For $k\geq1$, the class $\Sigma_k\R$
(resp.\ $\Pi_k\R$) consists of all decision problems that are polynomially
reducible to deciding the truth of prenex sentences
\[
Q_1 X_1\ Q_2 X_2\ \cdots\ Q_k X_k : F(X_1,\ldots,X_k),
\]
where each $X_i$ is a block of real variables, the quantifiers $Q_1,\ldots,Q_k$
alternate and begin with $\exists$ (resp.\ $\forall$), and $F$ is a
quantifier-free formula over polynomial (in)equalities with integer
coefficients as above. In particular $\Sigma_1\R=\exists\R$. It is
known~\cite{Shor1991Stretchability,existsRandPSPACE} that
\[
\NP \subseteq \exists\R \subseteq \PSPACE,
\]
and it is widely conjectured that $\NP \neq \exists\R$~\cite{schaefer2024existential};
consequently, $\exists\R$-hard problems are not expected to belong to~$\NP$. The
precise relationship between the real hierarchy and the polynomial hierarchy
remains unknown. The (full first-order) theory of the reals is
decidable and every fixed level $\Sigma_k\R$ is
contained in $\PSPACE$~\cite[Remark~13.10]{basu2006algorithms}.

\paragraph{Undecidability}
A decision problem is said to be undecidable if no algorithm can correctly
answer every instance with a YES-or-NO decision in finite time. In particular,
undecidable problems cannot belong to $\PSPACE$, nor to any other complexity
class consisting solely of decidable problems.

\paragraph{Quadratic optimization}
Quadratic constraints naturally arise in robust optimization problems with
decision-dependent uncertainty sets. Below, we define the quadratic
optimization problems relevant to our analysis and summarize their known
complexity classifications. Let $\mathcal{Z} \subseteq \R^{n_z}_{\ge 0}$ be a rational
polyhedron given by its outer representation, and let $Q_\ell \in
\mathbb{Q}^{n_z \times n_z}$ be symmetric matrices for $\ell \in [N]$.
Moreover, let $n_z = \nzc + \nzd$ and $R \in \mathbb{Q}_{\ge 0}$.
We consider the
following decision problems:
\begin{align}
    \label{eq:QMIP}\tag{\QMIP}
    &\exists z \in \mathcal{Z} \cap (\R^{\nzc} \times \Z^{\nzd}): z^\top Q_\ell z \leq 0, \; \forall \ell \in [N]? \\
    \label{eq:QIP}\tag{\QIP}
    &\exists z \in \mathcal{Z} \cap \Z^{n_z}: z^\top Q_\ell z \leq 0, \; \forall \ell \in [N]? \\
    \label{eq:QP}\tag{\QP}
    &\exists z \in \mathcal{Z}: z^\top Q_\ell z \leq 0, \; \forall \ell \in [N]? \\
    \label{eq:boundedQP}\tag{\boundedQP}
    &\exists z \in \mathcal{Z} \cap [0, R]^{n_z}: z^\top Q_\ell z \leq 0, \; \forall \ell \in [N]?
\end{align}
For the single-constraint variants, we define (where $Q \in \mathbb{Q}^{n_z \times n_z}$ is again a symmetric matrix):
\begin{align}
    \label{eq:QMIPsingle}\tag{\QMIPsingle}
    &\exists z \in \mathcal{Z} \cap (\R^{\nzc} \times \Z^{\nzd}): z^\top Q z \leq 0? \\
    \label{eq:boundedQMIPsingle}\tag{\boundedQMIPsingle}
    &\exists z \in \mathcal{Z} \cap (\R^{\nzc} \times \Z^{\nzd}) \cap [0, R]^{n_z}: z^\top Q z \leq 0?
\end{align}
As for the sets $\Xc$, $\Uc$, and $\Yc(x,u)$, the homogeneous, zero right-hand side form above is without loss of generality:
A general quadratic constraint $\tilde z^\top \tilde Q_\ell \tilde z + \tilde q_\ell^\top \tilde z \leq \tilde r_\ell$ can be
turned into one of the form $z^\top  Q_\ell z \leq 0$ by
introducing an auxiliary variable $z_{\nzt}$, appending the linear constraint
$z_\nzt = 1$ to $\mathcal Z$, and setting $z = (\tilde z, z_{\nzt})$.
Similarly, the nonnegativity of $z$ is without loss of generality, as each free variable can be split into the difference of two nonnegative ones; this preserves the complexity results below.
Moreover, note that \boundedQMIPsingle and \QMIPsingle encompass decision versions of optimization problems with a quadratic objective function and linear constraints.

The following theorems summarize the complexity of these problems.

\begin{theorem}[\cite{jeroslow1973there}]
    \label{thm:QIP}
    \QIP and \QMIP are undecidable.
\end{theorem}

\begin{theorem}[\cite{schaefer2024existential}]
    \label{thm:QP}
    \QP and \boundedQP are $\exists\R$-complete.
\end{theorem}

\begin{theorem}[\cite{pia2017mixed}]
    \label{thm:QMIPsingle}
    \boundedQMIPsingle and \QMIPsingle are $\NP$-complete.
\end{theorem}

\section{Complexity of {\static}} \label{sec:static}

In this section, we consider the complexity of \static and \boundedstatic, summarized in Table~\ref{tab:static}. We start by reformulating \static as the problem of verifying whether the following set is non-empty:
\begin{equation}
    \label{eq:Xstatic}
\Xstatic \equiv \set{x\in \X}{u^\top H_\ell x\leq 0,\; \forall\ell\in[L], u\in \U(x)}.
\end{equation}

We first consider continuous uncertainty sets, possibly depending on $x$. In this case, linear programming duality can be used to reformulate the constraints $u^\top H_\ell x\leq 0,\; \forall u\in \U(x)$ as follows:
\begin{align}
u^\top H_\ell x \leq 0,\; \forall u\in \Ucx &\Leftrightarrow
\max_{u\in\Ucx} u^\top H_\ell x \leq 0\label{eq:dualization:first} \\
&\Leftrightarrow
\min\set{\alpha_\ell^\top b}{B^\top(x)\alpha_\ell\geq H_\ell x,\; \alpha_\ell \in \R^{\muu}_{\ge 0}}\leq 0 \\
&\Leftrightarrow
\set{\alpha_\ell\in\R^{\muu}_{\ge 0}}{B^\top(x)\alpha_\ell\geq H_\ell x,\; \alpha_\ell^\top b\leq 0}\neq \emptyset\label{eq:staticdual}
\end{align}
Plugging~\eqref{eq:staticdual} into~\eqref{eq:Xstatic}, we obtain
\begin{equation}
    \label{eq:static_nonlinear}
\Xstatic_{\textup{lift}} \equiv \set{x\in \X, \alpha\in\R^{L\times\muu}_{\ge 0}}{B^\top(x)\alpha_\ell\geq H_\ell x,\; \alpha_\ell^\top b\leq 0,\;\forall \ell\in[L]},
\end{equation}
where $\Xstatic$ is non-empty if and only if $\Xstatic_{\textup{lift}}$ is non-empty. Indeed, $x\in\Xstatic$ if and only if there is some $\alpha\in\R^{L\times\muu}_{\ge 0}$ such that $(x,\alpha)\in\Xstatic_{\textup{lift}}$. Note that $\Xstatic_{\textup{lift}}$ is a polyhedron whenever $\X=\Xc$ and $B(x)=B$ for all $x\in\Xc$, showing the well-known result that $\static(\Xc,\Uc)$ is in $\P$.
Similarly, $\X=\Xd$ and $B(x)=B$ for all $x\in\Xd$ leads to a mixed-integer programming description for $\Xstatic$, which is in $\NP$ (as formalized in Fact~\ref{fact:milp}), and the $\NP$-hardness follows, e.g., by taking $\Xd$ as an integer programming description of the ($\NP$-hard) decision version of the vertex cover problem and $\Uc$ as the singleton $\{\one\}$. As both $\Xd$ and $\Uc$ are bounded, this reduction also proves the hardness of $\boundedstatic(\Xd,\Uc)$.

In the case where $B(x)$ is an arbitrary linear function of $x$,~\eqref{eq:static_nonlinear} is a set defined by polynomial constraints of degree at most 2, so the problem is in $\exists\R$ in case all variables are continuous. Additionally, it is not hard to see that~\eqref{eq:static_nonlinear} encompasses the quadratically-constrained sets used in the definitions of \QP and \QIP. Indeed, assuming $L=1$ (and removing the index $\ell$) and $\nxt=\muu+1$, the aforementioned reduction is easily seen by taking $2\muu$ constraints of $B^\top(x)\alpha\geq Hx$ (and the auxiliary variable $x_{\nxt}=1$) to impose $\alpha_i=x_i$ for $i\in[\muu]$, with the remaining constraints modeling the $N$ quadratic constraints present in \QIP or \QP, and setting $b=\zero$, as well as
the linear constraints of $\mathcal{Z}$.
From Theorems~\ref{thm:QIP} and~\ref{thm:QP}, we obtain the following two results.
\begin{corollary}
    \label{cor:XcUcx}
$\boundedstatic(\Xc,\Ucx)$ and $\static(\Xc,\Ucx)$ are $\exists\R$-complete.
\end{corollary}
\begin{corollary}
    \label{cor:XdUcx}
$\static(\Xd,\Ucx)$ is undecidable.
\end{corollary}
While Corollary~\ref{cor:XcUcx} illustrates hardness for both the bounded and unbounded cases, Corollary~\ref{cor:XdUcx} only focuses on the unbounded case. Indeed, in the bounded case, we next show that the problem is in $\NP$; the hardness with respect to this class follows from the hardness of the decision-independent case.
\begin{theorem}
    \label{thm:boundedXdUcx}
$\boundedstatic(\Xd,\Ucx)$ is in $\NP$.
\end{theorem}
\begin{proof}
Consider again the reformulation $\Xstatic_{\textup{lift}}$ introduced in~\eqref{eq:static_nonlinear} and let $(x^*,\alpha^*)$ be a YES answer. Since $x^*$ is integer and bounded by $R$, it has polynomial encoding length. Therefore, the projection onto $\alpha$ of the slice of $\Xstatic_{\textup{lift}}$ obtained by fixing $x$ to $x^*$, i.e.,
\begin{equation}
    \nonumber
\set{\alpha\in\R^{L\times\muu}_{\ge 0}}{B^\top(x^*)\alpha_\ell\geq H_\ell x^*,\; \alpha_\ell^\top b\leq 0,\;\forall \ell\in[L]},
\end{equation}
is a rational polyhedron with polynomial encoding length. By Fact~\ref{fact:milp}, it is non-empty if and only if it contains a vector of polynomial encoding length.
\end{proof}

Let us further discuss reformulation~\eqref{eq:static_nonlinear} in the case where $L=1$ and the decision-dependence appears only in the right-hand side, $b(x)$, leading to the special case
\begin{equation}
    \nonumber
\Xstaticb_{\textup{lift}} \equiv \set{x\in \X, \alpha\in\R^{\muu}_{\ge 0}}{B^\top\alpha \geq Hx,\; \alpha^\top b(x)\leq 0}.
\end{equation}
Only one of the constraints defining $\Xstaticb_{\textup{lift}}$ is quadratic, showing that the problem $\boundedstatic(\X,\Ucx)$ is in $\NP$ for $\X\in\{\Xc,\Xd\}$, thanks to Theorem~\ref{thm:QMIPsingle}. The same argument shows that $\static(\X,\Ucx)$ is in $\NP$ in this special case, since Theorem~\ref{thm:QMIPsingle} also covers the unbounded \QMIPsingle. In fact, relying on $B^\top\alpha\geq Hx$ to impose $\alpha_i=x_i$ for $i\in[\muu]$ as above, and choosing $b(x)$ appropriately, we obtain a reduction from \boundedQMIPsingle to \boundedstatic in this case, proving that $\boundedstatic(\Xd,\Ucx)$ with $L=1$ and right-hand side decision-dependence is $\NP$-complete.

The observation that $\static(\X,\Ucx)$ with $L=1$ and right-hand side decision-dependence only is in $\NP$ has an interesting implication in the realm of bilevel optimization (known to be related to decision-dependent robust optimization; see, for instance,~\cite{OJMO_2025__6__A2_0}). While bilevel linear optimization is known to be $\NP$-complete~\cite{jeroslow1985,Buchheim23}, we provide below a different proof for the fact that it is contained in $\NP$, based on a reduction from bilevel linear problems to $\static(\Xc,\Ucx)$ with $L=1$ and right-hand side decision-dependence. We start by recalling the formulation of the bilevel linear problem from~\cite{Buchheim23}:
\begin{equation}
    \label{eq:bilevel}
\min_{\check x, \check y}\set{\check c^\top \check x + \check d^\top \check y}{\check A \check x + \check B\check y = \check a, \check x\in\R^{\ncx}_{\ge 0}, \check y\in\argmin_{\bar y}\set{\check q^\top \bar y}{\check T \check x + \check W \bar y = \check h, \bar y \in\R^{\ncy}_{\ge 0}}},
\end{equation}
where all variables and parameters are marked with $\check \bullet$ to distinguish them from the notation used so far, $\check x\in\R^\ncx$ and $\check y\in\R^\ncy$ are the variables of the leader and of the follower, respectively, while upper case $\check A$, $\check B$, $\check T$, $\check W$ denote matrices of parameters, and $\check c$, $\check d$, $\check a$, $\check h$ denote vectors of parameters, all of conformal dimensions. Observe that~\eqref{eq:bilevel} is an optimistic bilevel problem because the follower's variables, $\check y$, are ultimately decided by the leader, who will then select a most favorable one among the follower's optimal solutions. Given a rational threshold value $V$, the decision version of~\eqref{eq:bilevel} can be written as:
\begin{equation}
    \label{eq:bileveldec}\tag{\bilevel}
    \exists (\check x,\check y)\in \R_{\ge 0}^{\ncx+\ncy} : \check A \check x + \check B\check y = \check a, \check y\in\argmin_{\bar y}\set{\check q^\top \bar y}{\check T \check x + \check W \bar y = \check h, \bar y \in\R^{\ncy}_{\ge 0}},\check c^\top \check x + \check d^\top \check y\leq V?
\end{equation}
\begin{theorem}
    \bilevel is in $\NP$.
\end{theorem}
\begin{proof}
Let us redefine the constraints $\check A \check x + \check B\check y = \check a$ so that they also contain $\check c^\top \check x + \check d^\top \check y\leq V$ (adding the corresponding slack variable as a leader's variable). The condition $\check y\in\argmin_{\bar y}\set{\check q^\top \bar y}{\check T \check x + \check W \bar y = \check h, \bar y \in\R^{\ncy}_{\ge 0}}$ holds if and only if $\check y$ is feasible for the follower, i.e., $\check T \check x + \check W \check y = \check h$ and $y \geq 0$, and $\check q^\top \check y\leq \check q^\top \bar y$ holds for every feasible $\bar y$. Hence, \bilevel can be reformulated through an additional quantifier and a linear inequality:
\begin{equation}
    \label{eq:bilevelreformulated}
\exists (\check x,\check y)\in \set{(\check x,\check y)\in\R_{\ge 0}^{\ncx+\ncy}}{\check A \check x + \check B\check y = \check a,\ \check T \check x + \check W \check y = \check h}\ \forall \bar y\in\set{\bar y\in\R_{\ge 0}^{\ncy}}
{\check T \check x + \check W \bar y = \check h}:\check q^\top \check y \leq \check q^\top \bar y?
\end{equation}
Using the correspondence $x=(\check x,\check y)$ and $u=\bar y$, we see that~\eqref{eq:bilevelreformulated} is indeed of the form of $\static(\Xc,\Ucx)$, for the decision-dependent uncertainty set 
$$
\U(\check x,\check y)=\set{\bar y\in\R_{\ge 0}^{\ncy}}{\check T \check x + \check W \bar y = \check h},
$$
for which the decision-dependence lies in the right-hand side $\check h -\check T \check x$. Observe that $\U(\check x,\check y)\neq\emptyset$ for every $(\check x,\check y)\in\X$, since it contains $\check y$. Hence, the reduction satisfies the standing assumption on uncertainty sets, and the result follows from the membership of $\static(\Xc,\Ucx)$ with $L=1$ and right-hand-side decision-dependence in $\NP$.
\end{proof}
The connection between bilevel and robust optimization used in the latter proof was also observed in \cite{OJMO_2025__6__A2_0}.

We now turn to discrete uncertainty sets and first consider the case of decision-independent uncertainty. While the reformulation of~\eqref{eq:static_nonlinear} is no longer valid here because $\U$ is discrete, we can rely on the original definition~\eqref{eq:Xstatic} and turn it finite by considering the extreme points and extreme rays of $\conv(\U)$:
\begin{align}
\Xstatic &=
\set{x\in\X}{
u^\top H_\ell x \leq 0,\; \forall \ell\in[L], u\in\U
}\\
&=
\set{x\in\X}{
u^\top H_\ell x \leq 0,\; \forall \ell\in[L],u\in\ext(\conv(\U))\cup\ray(\conv(\U))
},\label{eq:staticreformulatedtextreme}
\end{align}
where we know from Fact~\ref{fact:inthull} that all vertices and rays of $\conv(\U)$ have polynomially bounded encoding lengths. We next leverage the reformulation~\eqref{eq:staticreformulatedtextreme} to show that $\static(\Xc,\Ud)$ is $\coNP$-complete, and then, that $\static(\Xd,\Ud)\in\Sigma_2^p$.

\begin{theorem}
    \label{thm:static_coNPC}
    $\boundedstatic(\Xc,\Ud)$ and $\static(\Xc,\Ud)$ are $\coNP$-complete.
\end{theorem}
\begin{proof}
Let us first prove the membership of the unbounded version (which implies the one of the bounded version) in $\coNP$. For this, let $I$ be an instance of the problem, which is a YES-instance if and only if $\Xstatic$ is non-empty. In the case of $\X=\Xc$, the description of $\Xstatic$ from~\eqref{eq:staticreformulatedtextreme} can be rewritten as
\begin{equation}
    \label{eq:staticreformulatedtextreme_bis}
\Xstatic = \set{x\in\R^{\nxt}_{\ge 0}}{
\begin{aligned}
&Ax\leq a \\
&u^\top H_\ell x \leq 0,\; \forall \ell\in[L],u\in\ext(\conv(\U))\cup\ray(\conv(\U)) \\
\end{aligned}
}
\end{equation}
Applying Farkas' lemma to~\eqref{eq:staticreformulatedtextreme_bis}, with $\W=\ext(\conv(\U))\cup\ray(\conv(\U))$, the set $\Xstatic$ is non-empty if and only if
$$
a^\top \alpha \geq 0\mbox{ for all }\alpha \in \mathbb{R}_{\ge 0}^{\mx},\beta \in \mathbb{R}_{\ge 0}^{L \times \lvert \W \rvert}\mbox{ with }A^\top \alpha + \sum_{\ell\in[L],\atop u\in \W}H_\ell^\top u\,\beta_{\ell,u} \geq 0,
$$
which is equivalent to
\begin{equation}
\label{eq:static_coNP}
\min\set{a^\top \alpha}{A^\top \alpha + \sum_{\ell\in[L],\atop u\in \W}H_\ell^\top u\,\beta_{\ell,u} \geq 0,\; \alpha \in \mathbb{R}_{\ge 0}^{\mx},\beta \in \mathbb{R}_{\ge 0}^{L \times \lvert \W \rvert}}\geq 0.
\end{equation}
By the above equivalences, $I$ is a NO-instance to $\static(\Xc,\Ud)$ if and only if there exists a feasible solution to the linear program in~\eqref{eq:static_coNP} having a negative objective value. 
Since the feasibility set of~\eqref{eq:static_coNP} is a cone, there exists such a solution with negative objective value if and only if there is a solution with $a^\top\alpha=-1$. Adding this constraint to~\eqref{eq:static_coNP}, we obtain 
a linear program with only $\nxt+1$ constraints. Therefore, any basic solution $(\alpha^*,\beta^*)$ contains at most $\nxt+1$ non-zero variables. In particular, at most $\nxt+1$ components of $\beta^*$ may be positive. Moreover, since each $u\in\W$ has a polynomial encoding length, the value of each of the positive components of $(\alpha^*,\beta^*)$ also has a polynomial encoding length. Hence, if $I$ is a NO-instance, there exists a polynomial certificate of non-negative values $\alpha^*,\beta_{\ell_1,u_1}^*,\ldots,\beta_{\ell_{\nxt+1},u_{\nxt+1}}^*$ (all other $\beta_{\ell,u}^*$ being zero) such that $A^\top \alpha^* + \sum_{k\in[\nxt+1]}H_{\ell_k}^\top u_k\,\beta^*_{\ell_k,u_k}\geq 0$ and $a^\top \alpha^* <0.$
This shows that the complement of our problem is in $\NP$ and therefore our problem is in $\coNP$.

To prove the $\coNP$-hardness of the bounded version (which implies the one of the unbounded version), consider the problem with $L=1$, $\nxt=\nut=n+1$, auxiliary variables $x_{\nxt}=1$ and $u_{\nut}=1$, and $H=\left(\begin{smallmatrix} I_n & \zero \\ \zero^\top & -V \end{smallmatrix}\right)$, so that $(u',1)^\top H (x',1) = u'^\top x' - V$. We then have
\begin{align*}
&\exists x\in\{\one\}\ \forall u\in\U'\times\{1\}: u^\top H x \leq 0\\
\Leftrightarrow\quad&\exists x'\in\{\one\}\ \forall u'\in\U': u'^\top x' \leq V,
\end{align*}
the complement of which is
$$
\forall x'\in\{\one\}\ \exists u'\in\U': u'^\top x' > V,
$$
which can be simplified to
\begin{equation}
    \label{eq:co_static}
\exists u'\in\U': u'^\top \one > V.
\end{equation}
Problem~\eqref{eq:co_static} is easily seen to be $\NP$-hard by considering $\U'$ as the (bounded) feasibility set of the maximum clique problem, providing a reduction from the latter to the complement of $\boundedstatic(\Xc,\Ud)$, thereby proving its $\coNP$-hardness.
\end{proof}

Whenever decision variables are discrete, the $\Sigma_2^p$-hardness of the bounded problem is well-known~\cite{claus2020note}. The latter paper, while also addressing membership in $\Sigma_2^p$, however considers binary variables and uncertain parameters, thereby leaving open the membership in $\Sigma_2^p$ in the case where these are allowed to take integer values. We fill this small gap by again leveraging the reformulation~\eqref{eq:staticreformulatedtextreme} and relying on folklore results on the integer hull of polyhedra.

\begin{theorem}
    \label{thm:XcXdSigma}
$\static(\Xd,\Ud)$ is in $\Sigma_2^p$.
\end{theorem}
\begin{proof}
Reformulation~\eqref{eq:staticreformulatedtextreme} and Fact~\ref{fact:milp} imply that, if $\Xstatic$ is non-empty, it also contains a vector with polynomially bounded encoding length. Furthermore, reformulation~\eqref{eq:staticreformulatedtextreme} also shows that, given $x$, we have that $\forall \ell\in[L],u \in \U:u^\top H_\ell x\leq 0$ holds if and only if $\forall \ell\in[L],u\in \W:u^\top H_\ell x\leq0$, where $\W=\ext(\conv(\U))\cup\ray(\conv(\U))$, which contains only vectors of polynomially bounded encoding lengths by Fact~\ref{fact:inthull}.
\end{proof}

In the remainder of this section, we consider the two cases arising with discrete decision-dependent uncertainty sets, starting with $\static(\Xd,\Udx)$. Similarly to the continuous uncertainty case, we argue that the setting encompasses some quadratic problems. More precisely, consider the case where $\nxt=\nut$ and where $B$ and $b$ are such that the constraints $B(x)u\leq b$ become $x=u$ (using auxiliary variables $x_\nxt=1$ and $u_\nut=1$), leading to a decision problem of the form
\begin{equation}
    \label{eq:reformulationDDROquad}
\exists x\in\X:x^\top H_\ell x \leq 0,\;\forall \ell\in[L]?
\end{equation}
Problem~\QIP reduces to~\eqref{eq:reformulationDDROquad} for appropriate choices of $\Xd$ and $H_\ell$, so Theorem~\ref{thm:QIP} leads to undecidability in this case again.
\begin{corollary}
    \label{cor:XUdx}
$\static(\Xd,\Udx)$ is undecidable.
\end{corollary}

While \QIP is undecidable for unbounded variables, its bounded counterpart lies in $\NP$, so the above reduction does not bring additional insights into the complexity of $\boundedstatic(\Xd,\Udx)$.
In fact, when all variables are discrete and bounded, membership in $\Sigma_2^p$ is immediate.
\begin{observation}
    \label{obs:boundedSigma2p}
$\boundedstatic(\Xd,\Udx)$ is in $\Sigma_2^p$.
\end{observation}
Second, we consider the case of continuous decisions and discrete decision-dependent uncertainty, and we prove the membership as well as the hardness with respect to $\Sigma_2^p$ in the bounded case. Before proving the result, we introduce a small technical lemma which will be used in the following theorem as well as in the next section to simplify the handling of $\Udx$ in the bounded case. We state the lemma for \twostage, which is without loss of generality, as \static is the special case of \twostage obtained by imposing $y=\zero$ in $\Y(x,u)$.
\begin{lemma}
    \label{lem:Udx}
Let $I$ be an instance of \boundedtwostage, let $\Q\subseteq\R^\nxt$ be a polytope such that each of its inequalities has polynomial encoding length in
$\langle I \rangle$, and let $x^*\in\Q$. Then, the set $\Q^*\equiv\set{x\in \Q}{\Udx=\Ud(x^*)}$ contains a point~$\hat{x}$ of polynomial encoding length.
\end{lemma}
\begin{proof}
From the definition of $\Q^*$, we have
 \begin{align}
\Q^*=&\set{x\in\Q}{\forall u \in \Ud(x^*):B(x)u\leq b,\; \forall u \in \barUd(x^*)\exists i\in[\muu]:(B(x)u)_i> b_i}\\
=&\bigcup_{\phi\in[\muu]^{\barUd(x^*)}}
\set{x\in\Q}{B(x)u\leq b,\;\forall u \in \Ud(x^*),\; (B(x)u)_{\phi(u)}> b_{\phi(u)},\;\forall u \in \barUd(x^*)},\label{eq:UxequalUxstarunion}
\end{align}
where $\barUd(x^*)=\set{u\in\Z^\nut_{\ge 0}}{u\leq R}\setminus\Ud(x^*)$ denotes the complement of $\Ud(x^*)$ with respect to the set of integers bounded by $R$, and the union in~\eqref{eq:UxequalUxstarunion} is taken over all functions $\phi$ from $\barUd(x^*)$ to $[\muu]$. 
Since $x^*\in\Q^*$, it belongs to at least one of the sets considered in the union, so there is a function $\phi^*$ from $\barUd(x^*)$ to $[\muu]$ for which $x^*$ is in
\begin{equation}
\label{eq:varphistaropen}
\set{x\in\Q}{B(x)u\leq b,\;\forall u \in \Ud(x^*),\; (B(x)u)_{\phi^*(u)}> b_{\phi^*(u)},\;\forall u \in \barUd(x^*)}.
\end{equation}
Consider the closure of~\eqref{eq:varphistaropen}
\begin{equation}
\label{eq:varphistarclosed}
\set{x\in\Q}{B(x)u\leq b,\;\forall u \in \Ud(x^*),\; (B(x)u)_{\phi^*(u)}\geq b_{\phi^*(u)},\;\forall u \in \barUd(x^*)}.
\end{equation}
Set~\eqref{eq:varphistarclosed} is a polytope, and let us call its dimension $\rho\leq\nxt$. Observe that inequalities $(B(x)u)_{\phi^*(u)}\geq b_{\phi^*(u)},\;\forall u \in \barUd(x^*)$ cannot be implied equalities of the polytope as $x^*$ satisfies these inequalities strictly. Therefore, any point in the relative interior of the polytope satisfies these inequalities strictly as well. We construct the desired point $\hat{x}$ of polynomial encoding length in the relative interior by taking the centroid of $\rho+1$ affinely independent vertices.
\end{proof}

\begin{theorem}
    \label{thm:XcUdx}
$\boundedstatic(\Xc,\Udx)$ is $\Sigma_2^p$-complete.
\end{theorem}
\begin{proof}
Let us start with the membership and assume that $I$ is a YES-instance. Let $x^*$ be the corresponding certificate for the outer quantifier. We show next how to construct another certificate $\hat{x}$ having polynomial encoding length. Let us introduce the variant $\Xstatic_*$ of $\Xstatic$ considering the uncertainty set $\U(x^*)$ instead of the decision-dependent uncertainty set $\U(x)$, namely
$$
\Xstatic_* \equiv \set{x\in \X}{u^\top H_\ell x\leq 0,\; \forall\ell\in[L], u\in \U(x^*)}.
$$
We further consider the subset of $\Xstatic_*$ for which, additionally, $\U(x)=\U(x^*)$. We see that $x^*\in\set{x\in \Xstatic_*}{\U(x)=\U(x^*)}$, so applying Lemma~\ref{lem:Udx} to $\Q=\Xstatic_*$ yields the desired $\hat{x}$. To conclude the proof of the membership, we observe that, by the boundedness assumption, any $u\in\Udx$ has polynomial encoding length.

We proceed with the hardness and consider a reduction from the following bilevel knapsack problem, known to be $\Sigma_2^p$-hard~\cite[Theorem~3.1(c)]{caprara2014complexity}.
Given $a,b\in\Z_{\ge 0}^n$, $A,B\in\Z_{\ge 0}$ (we assume $A>0$ w.l.o.g.), and a threshold~$V\in\Z_{\ge 0}$, the decision version of the bilevel knapsack problem asks whether:
\begin{equation}
    \label{eq:bilevelknapsack}
    \exists X\in\set{X'\in\{0,1\}^n}{a^\top X' \leq A}\ \forall Y\in \set{Y'\in\{0,1\}^n}{b^\top Y'\leq B, X+Y' \leq \one}: b^\top Y \leq V.
\end{equation} 
The reduction defines the instance $I'$ corresponding to $I$ by setting $\X=\set{x\in[0,1]^n}{a^\top x \leq A}$,
$$
\U(x)=\set{u\in\{0,1\}^{\nut}}{b^\top u \leq B, x+u\leq (2 - \epsilon)\one},
$$
and $\epsilon = \frac{1}{2A}$. Then, we set $L=1$ and define $H_{1}=
\bigl(\begin{smallmatrix}
0 & b \\
0 & -V
\end{smallmatrix}\bigr)$, so
$$
\begin{pmatrix}u \\ 1 \end{pmatrix}^\top H_1 \begin{pmatrix}x \\ 1\end{pmatrix}\leq 0
\Leftrightarrow
b^\top u \leq V.
$$ 

First suppose that $I$ is a YES-instance for the bilevel knapsack problem. Let $X$ be the certificate to the outer quantifier. Setting $x=X$, the binarity of $x$ (and $u$) implies that $x+u\leq 2-\epsilon\Leftrightarrow x+u\leq 1$, so the projection of~$\U(x)$ on the variables $u$ is equal to $\set{Y'\in\{0,1\}^n}{b^\top Y'\leq B, X+Y' \leq 1}$. Hence, the satisfaction of $b^\top Y \leq V$ for all $Y$ in the latter set implies the satisfaction of $b^\top u\leq V$ for all $u$ in the aforementioned projection.

Now suppose that $I'$ is a YES-instance. Let $x$ be the certificate with the smallest number of fractional coefficients. If $x$ is binary, defining $X=x$ provides a YES certificate for $I$ in the same line as before. If $x$ has a fractional component $k$ such that $x_k\leq 1-\epsilon$, we can round this component to 0, obtaining a new certificate with fewer fractional components, yielding a contradiction. Thus, each fractional component of $x$ is larger than $1-\epsilon$. Therefore, we can define $\delta=\min\set{x_i}{i \in [n],\; x_i > 1-\epsilon}$ as the minimum fractional value of $x$. We round $x$ down to~$\hat{x}$ by defining $\hat{x}_i=x_i$ if $x_i\in\{0,1\}$ and $\hat{x}_i=\delta$ otherwise. Observe that $\U(x)=\U(\hat{x})$ and that $a^\top \hat{x}\leq A$, so $\hat{x}$ is a YES certificate if and only if $x$ is. Now consider the rounded-up solution $\hat{x}'=\lceil \hat{x} \rceil$ and observe that $\U(\hat{x}')=\U(\hat{x})$. Hence, if $a^\top \hat{x}'\leq A$, it is a YES certificate, providing a contradiction. Thus, $a^\top \hat{x}'\geq A+1$. Then, $a^\top \hat{x} > (1-\epsilon) a^\top \hat{x}'\geq (1-\epsilon)(A+1)\geq A.$
\end{proof}

We conclude this section by mentioning that the completeness of $\static(\Xc,\Udx)$ for $\Sigma_2^p$ is unknown in the unbounded case.
\begin{question}
Is $\static(\Xc,\Udx)$ in $\Sigma_2^p$?
\end{question}

\section{Complexity of \twostage}\label{sec:twostage}

We consider next the complexity of \twostage and \boundedtwostage summarized in Table~\ref{tab:twostage}. Observe first that any instance of \static reduces to an instance of \twostage by defining $\Y(x,u)=\set{0}{u^\top H_\ell x\leq 0,\;\forall \ell\in[L]}$.
\begin{observation}\label{obs:staticleqtwostage}
For each $\X\in\{\Xc,\Xd\}$, $\U\in\{\Uc,\Ud,\Ucx,\Udx\}$, and $\Y\in\{\Yc,\Yd\}$, it holds that $\static(\X,\U)\leq\twostage(\X,\U,\Y)$ and $\boundedstatic(\X,\U)\leq\boundedtwostage(\X,\U,\Y)$.
\end{observation}
Next, we start our study with continuous recourse and decision-independent uncertainty. Introducing the polyhedral cone $\Pi\equiv \set{\pi\in\R^\my_{\ge 0}}{W^\top\pi \geq 0}$, we can formulate the following counterpart of $\Xstatic$ in the two-stage setting.
\begin{lemma}
    \label{lem:Xtwo}
    If $\Y=\Yc$, a vector $x$ is a certificate for the outer existential quantifier if and only if $x\in\Xtwostage$ where
   $$
   \Xtwostage\equiv\set{x\in\X}{
    (T(u)x)^\top \pi \leq 0,\;\forall w=(u,\pi)\in\W
   }
   $$
where
$
\W=(\ext(\conv(\U))\cup\ray(\conv(\U)))\times\ray(\Pi).
$ 
\end{lemma}
\begin{proof}
Given $x\in\X$ and $u\in\U$ and applying Farkas' lemma to the linear system defining $\Y(x,u)$, there exists $y\in\Y(x,u)$ if and only if 
\begin{equation}
    \label{eq:innermostfeas}
(T(u)x)^\top\pi \leq 0, \quad \forall  \pi\in \Pi.
\end{equation}
Then, given $x\in\X$, \twostage asks whether~\eqref{eq:innermostfeas} holds for all $u\in\U$
\begin{equation}
    \label{eq:innermostfeasforallU}
(T(u)x)^\top\pi \leq 0, \quad \forall  \pi\in \Pi, u\in \U.
\end{equation}
We then decompose~\eqref{eq:innermostfeasforallU} in terms of extreme points and rays of the involved sets.
\end{proof}

The reformulation of Lemma~\ref{lem:Xtwo} will be used in the following two results.

\begin{theorem}
    \label{thm:inSigma2p}
    $\twostage(\Xd,\U,\Yc)$ is in $\Sigma_2^p$ for $\U\in\{\Uc,\Ud\}$.
\end{theorem}
\begin{proof}
From Lemma~\ref{lem:Xtwo}, the problem can be reformulated as $\exists x\in\X\ \forall w=(u,\pi)\in\W:(T(u)x)^\top \pi \leq 0$. It remains to show that we can restrict ourselves to a subset of $\X$ containing only vectors of polynomial input lengths. Observe that the relevant set of certificates for the outer quantifier is given by $\Xtwostage\subseteq\X$. Then, we know by Fact~\ref{fact:inthull} that all vertices and rays of $\conv(\U)$ and $\Pi$ have polynomially bounded encoding length. Hence, Fact~\ref{fact:milp} implies that if $\Xtwostage$ is non-empty, it also contains a vector $\hat x$ with polynomially bounded encoding length. Let $\hat \X\subseteq \X$ be all such vectors. Then, \twostage can be reformulated as
$
\exists x\in\hat \X,\forall w=(u,\pi)\in\W:(T(u)x)^\top \pi \leq 0.
$
\end{proof}

The $\Sigma_2^p$-hardness of $\boundedtwostage(\Xd,\Uc,\Yc)$ has been proved in~\cite{lefebvre2025exact}.

The $\Sigma_2^p$-hardness of $\boundedtwostage(\Xd,\Ud,\Yc)$ can be easily obtained from the hardness of $\static(\Xd,\Ud)$ by assuming $y=\zero$ for all $y\in\Y(x,u)$.

\begin{observation}
\label{obs:XdUd2stages}
$\boundedtwostage(\Xd,\Ud,\Yc)$ is $\Sigma_2^p$-hard.
\end{observation}

The next result considers the case of continuous first-stage variables. We remark that~\cite{chekuri2007hardness} prove the $\coNP$-hardness of the separation problem 
$$\forall u\in \U(x) \ \exists y \in\Y(x,u)?$$
 and conclude from that hardness and the ellipsoid algorithm that the optimization problem is also $\coNP$-hard. However, the $\coNP$-hardness of the separation problem does not imply the $\coNP$-hardness of the optimization problem under many-one reduction; in fact, the $\coNP$-hardness of the separation problem only proves that the optimization problem is not in $\P$, unless $\coNP=\P$. We fill this small gap below and prove that $\twostage(\Xc,\U,\Yc)$ is indeed $\coNP$-hard (more precisely, $\coNP$-complete) using a different reduction, complementing the hardness already proved in Theorem~\ref{thm:static_coNPC} for the case $\U=\Ud$.

\begin{theorem}
    \label{thm:2stagescoNPC}
    $\twostage(\Xc,\U,\Yc)$ and $\boundedtwostage(\Xc,\U,\Yc)$ are $\coNP$-complete for $\U\in\{\Uc,\Ud\}$.
\end{theorem}
\begin{proof}
Applying Farkas' lemma to $\Xtwostage$ with $\X$ replaced by its polyhedral description $\set{x\in\R^\nxt_{\ge 0}}{A x \leq a}$, $I$ is a YES-instance to \twostage if and only if
$$
a^\top \alpha  \geq 0, \quad \forall (\alpha,\beta)\in\set{(\alpha,\beta)\in\mathbb{R}_{\ge 0}^{\mx+|\W|}}{A^\top \alpha + \sum_{(u,\pi)\in\W}T(u)^\top\pi\cdot\beta_{u,\pi} \geq 0},
$$
which is equivalent to
\begin{equation}
    \label{eq:2roDEC:dual}
\min_{\alpha, \beta}\set{a^\top \alpha 
}{A^\top \alpha + \sum_{(u,\pi)\in\W}T(u)^\top\pi\cdot\beta_{u,\pi} \geq 0, \alpha\in\R^{\mx}_{\ge 0}, \beta \in\R^{|\W|}_{\ge 0}}\geq 0.
\end{equation}
Then, as in the proof of Theorem~\ref{thm:static_coNPC}, $I$ is a NO-instance to \twostage if and only if there exists a (basic) solution to the linear program in~\eqref{eq:2roDEC:dual} having a negative objective value. As before, we add without loss of generality the normalizing constraint $a^\top \alpha = -1$ and since the resulting linear program has only $\nxt+1$ constraints, any basic solution contains at most $\nxt+1$ non-zero variables. Hence, if $I$ is a NO-instance, there exists a polynomial certificate $\alpha^*,\beta_{u_1,\pi_1}^*,\ldots,\beta_{u_{\nxt+1},\pi_{\nxt+1}}^*$ (all other $\beta_{u,\pi}^*$ being zero) such that $A^\top \alpha^* + \sum_{k=1}^{\nxt+1}T(u_k)^\top\pi_k\cdot\beta_{u_{k},\pi_{k}}^* \geq 0$ and $a^\top\alpha^*=-1$.

To prove the hardness, consider the problem
$$
\exists x\in\R_{\ge 0}^{\nxt}\ \forall u\in \U\ \exists y \in \set{y\in\R_{\ge 0}^{\nyt}}{-y\leq u\leq y, \one^\top y \leq x}: x\leq V.
$$
The complementary problem is
\begin{equation}
\label{eq:negatedproblem}
\forall x\ \exists u\in \U\ \forall y \in \set{y\in\R_{\ge 0}^{\nyt}}{-y\leq u\leq y, \one^\top y \leq x}: x > V.
\end{equation}
Observe that the inequality is satisfied for all $x$ and $y$ if and only if it is the case when $y_i=|u_i|$ for each $i\in[\nyt]$ and $x=\one^\top y$, respectively, so~\eqref{eq:negatedproblem} can be rewritten as
\begin{equation}
\label{eq:2stagescoNPhard}
\exists u\in \U:\|u\|_1>V.
\end{equation}
When $\U$ is a binary set, $\|u\|_1=u^\top \one$ so the hardness follows considering $\U$ as the feasibility set of the maximum clique problem, as in the proof of Theorem~\ref{thm:static_coNPC}. When $\U$ is a polytope, $\|u\|_1=u^\top\one$ is linear because $\U\subseteq\R^{\nut}_{\ge 0}$, so we shift the variables. Let $P\subseteq[-1,1]^{n}$ be a polytope over which maximizing the $\ell_1$-norm is $\NP$-hard, see~\cite{guslitser2002uncertainty} as well as earlier works~\cite{freund1985complexity}. We define the polytope $\U=\set{(u,u_{n+1})\in[0,2]^{n}\times\{1\}}{u-\one\in P}$ and replace the constraints $-y\leq u\leq y$ by $-y\leq u-\one u_{n+1}\leq y$. The same argument as above shows that the complement of the resulting problem is $\exists u\in\U:\|u-\one\|_1>V$, which is $\NP$-hard.
\end{proof}
The hardness of the problem with discrete recourse and decision-independent uncertainty has mostly been covered in the literature. Specifically, the authors of~\cite{GoerigkLW24} address the bounded cases of all optimization variables being discrete, with either continuous or discrete uncertainty. The case of first-stage continuous variables can be derived from these results by adding dummy recourse variables $z$ together with the inclusion of constraints $x=z$ in the definition of $\Y(x,u)$.
\begin{observation}
    \label{obs:calld_hard}
$\boundedtwostage(\Xc,\U,\Yd)$ is $\Sigma_3^p$-hard for $\U\in\{\Uc,\Ud\}$.
\end{observation}
While the membership of these problems for $\Sigma_3^p$ is still open in the unbounded case, we prove it for most of the bounded settings. The membership of $\boundedtwostage(\Xd,\Ud,\Yd)$ (in fact, $\boundedtwostage(\Xd,\Udx,\Yd)$) for $\Sigma_3^p$ follows immediately from the definition. 
\begin{observation}
    \label{obs:XdUdYdinSigma}
$\boundedtwostage(\Xd,\Udx,\Yd)$ is in $\Sigma_3^p$.
\end{observation}
We focus next on two more involved cases.
\begin{theorem}
    \label{thm:XcUdYdinSigma}
    $\boundedtwostage(\Xc,\Ud,\Yd)$ is in $\Sigma_3^p$.
\end{theorem}
\begin{proof}
Assume there is a YES-answer to the non-emptiness of
\begin{equation}
    \nonumber
\set{(x,y)\in\R_{\ge 0}^\nxt\times\Z_{\ge 0}^{|\Ud|\times \nyt}}{Ax \leq a,\; T(u)x + W y(u) \leq 0, \; \forall u\in \Ud},
\end{equation}
and let us denote by $x^*,y^*$ this feasible point. Then, $x^*$ belongs to the polyhedron
\begin{equation}
    \nonumber
\set{x\in\R_{\ge 0}^\nxt}{Ax \leq a,\; T(u)x + W y^*(u) \leq 0, \; \forall u\in \Ud},
\end{equation}
the vertices of which have polynomial encoding lengths, because each $u$ and $y^*(u)$ has polynomial encoding length. Hence, any vertex provides an outermost certificate, and the two remaining quantifiers are defined over sets containing only elements of polynomially bounded encoding lengths.
\end{proof}

\begin{theorem}
\label{thm:XdUcYdinSigma}
$\boundedtwostage(\Xd,\Ucx,\Yd)$ is in $\Sigma_3^p$.
\end{theorem}
\begin{proof}
  By definition, any $x\in\Xd$ has a polynomial encoding length. Hence, we must show that, given $x$, the outermost certificate of the problem $\forall u \in \Uc(x)\ \exists y \in \Yd(x, u)$ can be restricted to vectors of polynomial encoding lengths. We proceed by showing that any certificate for a YES answer to the complement problem, 
  \begin{equation}
  \label{eq:negatedadversarial}
  \exists u \in \Ucx : \Yd(x, u) = \emptyset, 
  \end{equation}
  has polynomial encoding length. Let $\Y\equiv [0,R]^\nyt\cap\Z_{\ge 0}^\nyt$ and observe that any given $y\in\Y$ is not in $\Yd(x,u)$ if it violates one of the constraints of $\Yd(x,u)$, so Problem~\eqref{eq:negatedadversarial} can be reformulated as
\begin{align*}
  &\exists u \in \Ucx\ \forall y \in \Y\ \exists i \in [m_y] : T_i(u) x + W_i y > 0 \\
  \Leftrightarrow\quad &\exists u \in \Ucx : \min_{y \in \Y} \max_{i \in [m_y]} (T_i(u) x + W_i y) > 0
\end{align*}
Function $f(u)=\min_{y \in \Y} \max_{i \in [m_y]} (T_i(u) x + W_i y)$ is a piecewise linear function in $u$, defined on the domain $\Ucx$, so its maximum is reached at a vertex of the graph of $f$. Any such vertex is the unique solution of a linear system involving the pieces of $f$ and the linear constraints defining $\Ucx$, all of which have polynomial encoding length since $x$ has polynomial encoding length.
This proves that any YES-instance to~\eqref{eq:negatedadversarial} has a certificate with polynomial encoding length. Finally, for all $x$ and $u$, each $y\in\Y(x,u)$ has a polynomial encoding length, proving the result.
\end{proof}

In the rest of the section, we discuss the remaining cases with decision-dependent uncertainty. Let us start with the unbounded problems. We already observed in Observation~\ref{obs:staticleqtwostage} that $\static(\X,\U(x))$ reduces to $\twostage(\X,\U(x),\Y)$. Together with the undecidability of $\static(\Xd,\U(x))$, this implies the following result.
\begin{observation}
\label{obs:undecidableeasy}
$\twostage(\Xd,\U(x),\Y)$ is undecidable for $\Y\in\{\Yc,\Yd\}$.
\end{observation}
For $\X=\Xc$ and $\Y=\Yd$, integrality of $x$ can be imposed through integer recourse variables. Simply relaxing $\Xd$ to $\Xc$ in an instance of $\static(\Xd,\U(x))$ is however not sufficient in general, as $\U(x)$ could be empty for fractional $x$. We therefore reduce directly from \QIP.
\begin{theorem}
\label{thm:XcUxYdundecidable}
$\twostage(\Xc,\U(x),\Yd)$ is undecidable for $\U(x)\in\{\Ucx,\Udx\}$.
\end{theorem}
\begin{proof}
We reduce \QIP to $\twostage(\Xc,\U(x),\Yd)$ for both $\U(x)\in\{\Ucx,\Udx\}$. Given an instance of \QIP, we let $\X=\mathcal{Z}$, and we define
$$
\Uc(x)\equiv\set{(u,w)\in\R^{\nxt}_{\ge 0}\times[0,1]}{u=wx},
\qquad
\Y(x,(u,w))\equiv\set{y\in \Z^\nxt_{\ge 0}}{y=x,\;u^\top Q_\ell x \leq 0,\; \forall \ell \in [N]},
$$
and $\Udx\equiv\Ucx\cap\Z^{\nxt+1}$.
If $z$ is a YES-certificate of \QIP, let $x=z$ be an outer certificate for $\twostage(\Xc,\U(x),\Yd)$. 
Then, $y=x$ is trivially feasible for $(u,w)=(\zero,0)$ and it is also feasible for any $(\zero,0)\neq(u,w)\in\Uc(x)$ because $u^\top Q_\ell x=w\,x^\top Q_\ell x\leq 0$. Conversely, let $x$ be a YES-certificate of the constructed instance of $\twostage(\Xc,\U(x),\Yd)$. 
Observe that $(\zero,0)\in\Ud(x)\subseteq\Ucx$, leading to the constraints $y=x$ for this scenario, implying that $x$ is integer. Hence, scenario $(x,1)$ also pertains to $\Ud(x)\subseteq\Uc(x)$, and this scenario enforces $x^\top Q_\ell x\leq 0$ for all $\ell\in[N]$, so $x$ is a YES-certificate for \QIP.
\end{proof}
The complexity of the remaining two cases, $\twostage(\Xc,\Ucx,\Yc)$ and $\twostage(\Xc,\Udx,\Yc)$, is more difficult to characterize. They inherit the hardness from $\static(\Xc,\Ucx)$ and $\twostage(\Xc,\Uc,\Yc)$ (resp. $\static(\Xc,\Udx)$ and $\twostage(\Xc,\Ud,\Yc)$) but it is unknown whether they actually belong to the corresponding complexity classes. For $\twostage(\Xc,\Ucx,\Yc)$, we can use Farkas' lemma as in the proof of Lemma~\ref{lem:Xtwo} to reformulate the problem as the $\Sigma_2\mathbb{R}$ sentence
$$
\exists x\ \forall \pi,u:x\in\Xc\wedge\left(\left(\pi\in\Pi\wedge u\in\Ucx\right)\implies (T(u)x)^\top \pi \leq 0\right)?
$$
\begin{observation}
    \label{obs:XcUcxYc}
$\twostage(\Xc,\Ucx,\Yc)$ is in $\Sigma_2\R$.
\end{observation}
While we do not currently know whether $\twostage(\Xc,\Ucx,\Yc)$ is complete for $\Sigma_2\R$, we can prove hardness for the corresponding class in the polynomial hierarchy, $\Sigma_2^p$. The proof of the following result borrows ingredients from the proofs of Theorem~\ref{thm:XcUdx} and of the $\Sigma_2^p$-hardness of $\twostage(\Xd,\Uc,\Yc)$ from~\cite{lefebvre2025exact}.
\begin{theorem}
    \label{thm:XcUcxYcH}
$\twostage(\Xc,\Ucx,\Yc)$ is $\Sigma_2^p$-hard.
\end{theorem}
\begin{proof}
We reduce from the bilevel knapsack problem~\eqref{eq:bilevelknapsack} introduced in the proof of Theorem~\ref{thm:XcUdx}, known to be $\Sigma_2^p$-hard~\cite{caprara2014complexity}. For a binary $X$ with $a^\top X\leq A$, let us write $\Phi(X)\equiv\max\set{b^\top u}{u\in\{0,1\}^n,\ b^\top u\leq B,\ u\leq\one-X}$, so that~\eqref{eq:bilevelknapsack} is a YES-instance if and only if there exists a feasible binary $X$ with $\Phi(X)\leq V$.
 
We set $\epsilon=\frac{1}{2A}$ and let $M$ be an integer satisfying $M>\max_i(b_i^2-b_i)$ and $M>\max_i b_i(2A-1)$, for instance $M=\max_i b_i^2+2A\max_i b_i$, which has polynomial encoding length. The reduction defines the instance $I'$ through the continuous sets $\X=\set{x\in[0,1]^n}{a^\top x\leq A}$,
$$
\U(x)=\set{u\in[0,1]^n}{b^\top u\leq B,\ x+u\leq(2-\epsilon)\one},
$$
which are exactly the sets used in Theorem~\ref{thm:XcUdx}, except that $u$ now ranges over $[0,1]^n$ instead of $\{0,1\}^n$; the decision-dependence is carried by the constraint $x+u\leq(2-\epsilon)\one$. We introduce recourse variables $y$ that incentivize $u$ to be binary, by relaxing the constraint $b^\top u\leq V$ by an amount proportional to the fractionality of $u$:
$$
\Y(x,u)=\set{y\in\R_{\ge 0}^n}{y\leq u,\ y\leq\one-u,\ b^\top u-M\one^\top y\leq V},
$$
where $y$ quantifies the fractionality of $u$; in particular, if $u\in\{0,1\}^n$, then $y=\zero$. Observe that $\Y$ is a polyhedron of type $\Yc$ that does not depend on $x$.
 
We first reformulate the recourse. The first two constraints of $\Yc(x,u)$ are always feasible (for $y=\zero$). Hence, since $\max\set{\one^\top y}{0\leq y\leq u,\ y\leq\one-u}=\sum_i\min(u_i,1-u_i)$, the set $\Y(x,u)$ is non-empty if and only if $M\sum_i\min(u_i,1-u_i)\geq b^\top u-V$, that is,
\begin{equation}
    \label{eq:recourseequivXcUcxYc}
    \Y(x,u)\neq\emptyset\Longleftrightarrow f(u)\leq V,
\end{equation}
where $f(u)\equiv b^\top u-M\sum_i\min(u_i,1-u_i).$ Hence $I'$ is a YES-instance if and only if there exists $x\in\X$ with $g(x)\leq V$, where $g(x)\equiv\max_{u\in\U(x)}f(u)$. The function $f$ is convex, as each $-M\min(u_i,1-u_i)$ is convex, and does not depend on $x$; therefore $g$ depends on $x$ only through the feasible region $\U(x)$ and is monotone, in the sense that $\U(x')\subseteq\U(x)$ implies $g(x')\leq g(x)$.
 
We next show that for each binary $x\in \X$, there exists a maximizer $u^*$ of $f$ over $\U(x)$ that is binary.
Let $F=\set{i}{x_i=0}$ denote the \emph{free} coordinates. On a \emph{blocked} coordinate ($x_i=1$, hence $u_i\in[0,1-\epsilon]$), the contribution $f_i(u_i)\equiv b_iu_i-M\min(u_i,1-u_i)$ is convex with $f_i(0)=0$ and $f_i(1-\epsilon)=b_i(1-\epsilon)-M\epsilon<0$, using $M>b_i(1-\epsilon)/\epsilon=b_i(2A-1)$; hence $f_i\leq0$ on $[0,1-\epsilon]$. Since $f(u)=\sum_i f_i(u_i)$ is separable, setting all blocked coordinates of any $u\in\U(x)$ to $0$ preserves feasibility and does not decrease $f$. Therefore, the maximum of $f$ over $\U(x)$ equals its maximum over the face
$$
\U_F(x)\equiv\set{u\in\U(x)}{u_i=0,\ \forall i\notin F},
$$
which, restricted to the coordinates in $F$, is the knapsack polytope $\set{u\in[0,1]^F}{b^\top u\leq B}$. As $f$ is convex, this maximum is attained at an extreme point $u^*$ of $\U_F(x)$. Such an extreme point has at most one fractional component $i$, in which case the knapsack constraint is tight and $u_i^*=k/b_i$ for some integer $0<k<b_i$ (so $b_i\geq2$); therefore $\min(u_i^*,1-u_i^*)\geq1/b_i$ and $u_i^*\leq(b_i-1)/b_i$. Comparing with the solution $u'$ identical to $u^*$ except that $u'_i=0$,
$$
f(u^*)-f(u')=b_iu_i^*-M\min(u_i^*,1-u_i^*)\leq(b_i-1)-\frac{M}{b_i}<0
$$
by $M>b_i(b_i-1)$, which contradicts the optimality of $u^*$. Hence $u^*$ is binary, so $\sum_i\min(u^*_i,1-u^*_i)=0$, and
\begin{equation}
    \label{eq:binaryequal}
g(x)=\max\set{b^\top u^*}{u^*\in\{0,1\}^n,\ b^\top u^*\leq B,\ x+u^*\leq(2-\epsilon)\one}=\Phi(x),
\end{equation}
the last equality holding because, for binary $x$ and $u^*$, the constraint $x+u^*\leq(2-\epsilon)\one$ is equivalent to $u^*\leq\one-x$.
 
Suppose first that~\eqref{eq:bilevelknapsack} is a YES-instance, and let $X$ be the corresponding certificate, so that $X\in\{0,1\}^n$, $a^\top X\leq A$, and $\Phi(X)\leq V$. Setting $x=X\in\X$, the previous paragraph yields $g(X)=\Phi(X)\leq V$, that is, $\forall u\in\U(X)\ \exists y\in\Y(X,u)$ by~\eqref{eq:recourseequivXcUcxYc}. Thus $I'$ is a YES-instance.

Consider now a YES-instance $I'$ and let $x\in\X$ satisfy $g(x)\leq V$, and recall that $x$ is not necessarily binary. We round $x$ to the binary vector $X^*$ defined by $X^*_i=1$ if $x_i>1-\epsilon$ and $X^*_i=0$ otherwise, and we let $F=\set{i}{x_i>1-\epsilon}$. For a rounded-down coordinate ($x_i\leq1-\epsilon$), $\min(1,2-\epsilon-x_i)=1=\min(1,2-\epsilon-0)$, so the bound on $u_i$ stays $1$; for a rounded-up coordinate ($x_i>1-\epsilon$, $X^*_i=1>x_i$), the bound drops from $2-\epsilon-x_i\in(1-\epsilon,1)$ to $1-\epsilon$. Hence $\U(X^*)\subseteq\U(x)$, and monotonicity gives $g(X^*)\leq g(x)\leq V$. Moreover $X^*$ is feasible: as in the rounding argument of Theorem~\ref{thm:XcUdx},
$$
A\geq a^\top x\geq\sum_{i\in F}a_ix_i>(1-\epsilon)\sum_{i\in F}a_i=(1-\epsilon)\,a^\top X^*,
$$
so $a^\top X^*<\frac{A}{1-\epsilon}=\frac{2A^2}{2A-1}=A+\frac{A}{2A-1}\leq A+1$, and integrality yields $a^\top X^*\leq A$. By~\eqref{eq:binaryequal}, $\Phi(X^*)=g(X^*)\leq V$, so $X^*$ certifies~\eqref{eq:bilevelknapsack}.
 
The two implications show that $I'$ is a YES-instance if and only if~\eqref{eq:bilevelknapsack} is, and the construction is clearly polynomial. Since $\X$, $\U(x)$ and $\Y(x,u)$ are all contained in $[0,1]$-boxes (indeed $y\leq\min(u,\one-u)\leq\one$), the reduction produces bounded instances, which establishes the $\Sigma_2^p$-hardness of $\boundedtwostage(\Xc,\Ucx,\Yc)$, and therefore of $\twostage(\Xc,\Ucx,\Yc)$.
\end{proof}

Let us now study the bounded cases, and start with an easy case: given $x\in \X_d$ and $u\in\Udx$, then the bounded $\Yc(x,u)$ has polynomially bounded facet complexity, and so are its vertices, thanks to Fact~\ref{fact:inthull}. Hence, the problem $\boundedtwostage(\Xd,\Udx,\Yc)\in\Sigma_3^p$. We can refine this analysis by using an argument similar to that of Lemma~\ref{lem:Xtwo} for decision-dependent uncertainty sets. Specifically, given $x\in\X$ and $u\in\U$, there exists $y\in\Y(x,u)$ if and only if 
$(T(u)x)^\top\pi \leq 0$ for all $\pi\in \ray(\Pi)$. Hence, $\boundedtwostage(\X,\U(x),\Yc)$ can be cast as 
\begin{equation}
    \label{eq:twostageDD}
\exists x\in\X\ \forall (u,\pi)\in\ext(\conv(\U(x)))\times\ray(\Pi): (T(u)x)^\top\pi \leq 0?
\end{equation}
If $\X=\Xd$, all $x\in\Xd$ have polynomial encoding lengths, and the same holds for all $(u,\pi)\in\ext(\conv(\U(x)))\times\ray(\Pi)$ for both $\U(x)=\Ucx$ and $\U(x)=\Udx$.
\begin{observation}
    \label{obs:XdUxYc}
    $\boundedtwostage(\Xd,\U,\Yc)$ is in $\Sigma_2^p$ for $\U\in\{\Ucx,\Udx\}$.
\end{observation}
We further leverage~\eqref{eq:twostageDD} in the next theorem for the case where $\X=\Xc$, for which we also need to prove that small certificates exist. Observe that in this case, the hardness follows from Theorem~\ref{thm:XcUdx}.
\begin{theorem}
    \label{thm:XcUdYcDD}
    $\boundedtwostage(\Xc,\Ud(x),\Yc)$ is in $\Sigma_2^p$.
\end{theorem}
\begin{proof}
Let $x^*$ be an outer certificate for a YES-instance $\boundedtwostage$. Then, $x^*$ is in the set
   $$
   \Xtwostage_*\equiv\set{x\in\X}{
    (T(u)x)^\top \pi \leq 0,\;\forall w=(u,\pi)\in\W(x^*)
    },
   $$
where $\W(x^*)=\ext(\conv(\U(x^*)))\times\ray(\Pi)$. We see that $x^*\in\set{x\in\Xtwostage_*}{\Udx=\Ud(x^*)}$, so we apply Lemma~\ref{lem:Udx} to $\Q=\Xtwostage_*$ to obtain the certificate $\hat{x}$ of polynomial input length. Let $\X^*\subseteq\X$ be the corresponding subset of vectors of polynomial input length. The result then follows by considering the formulation~\eqref{eq:twostageDD} with $\X$ replaced by $\X^*$.
\end{proof}
Next, we consider the case of discrete recourse, leveraging once again Lemma~\ref{lem:Udx}.
\begin{theorem}
    \label{thm:XcUdYdxinSigma}
    $\boundedtwostage(\Xc,\Udx,\Yd)$ is in $\Sigma_3^p$.
\end{theorem}
\begin{proof}
Assume there is a YES-answer to the non-emptiness of
\begin{equation}
    \nonumber
\set{(x,y)\in\R_{\ge 0}^\nxt\times\Z_{\ge 0}^{|\Ud|\times \nyt}}{Ax \leq a,\; T(u)x + W y(u) \leq 0, \; \forall u\in \Udx},
\end{equation}
and let us denote by $x^*,y^*$ this feasible point. Then, $x^*$ belongs to the polyhedron
\begin{equation}
    \nonumber
\Xtwostage_*\equiv\set{x\in\R_{\ge 0}^\nxt}{Ax \leq a,\; T(u)x + W y^*(u) \leq 0, \; \forall u\in \Ud(x^*)}.
\end{equation}
We see that $x^*\in\set{x\in\Xtwostage_*}{\Udx=\Ud(x^*)}$, so we apply Lemma~\ref{lem:Udx} to $\Q=\Xtwostage_*$ to obtain the polynomial certificate~$\hat{x}$.
\end{proof}

Last, for $\boundedtwostage(\Xc,\Ucx,\Yd)$, since $y$ is bounded, its integrality can be expressed through polynomially many binary expansions $y_j=\sum_{s}2^s b_{js}$ with $b_{js}(b_{js}-1)=0$, which leads to a sentence of the form $\exists x\ \forall u\ \exists y,b$.
\begin{observation}
    \label{obs:XcUcxYdinS3R}
$\boundedtwostage(\Xc,\Ucx,\Yd)$ is in $\Sigma_3\R$.
\end{observation}

We conclude the section by listing the open questions for \twostage.

\begin{question}
For each $\X\in\{\Xc,\Xd\}$ and $\U\in\{\Uc,\Ud\}$, is $\twostage(\X,\U,\Yd)$ in $\Sigma_3^p$? (also in the bounded case for $\Xc$ and $\Uc$)
\end{question}

\begin{question}
Are $\boundedtwostage(\Xc,\Uc(x),\Yc)$ and $\twostage(\Xc,\Uc(x),\Yc)$ in $\Sigma_2^p$?
\end{question}

\begin{question}
Is $\twostage(\Xc,\Ud(x),\Yc)$ in $\Sigma_2^p$?
\end{question}

\begin{question}
Is $\boundedtwostage(\Xc,\Uc(x),\Yd)$ in $\Sigma_3^p$?
\end{question}

\section{Complexity of \Kadapt{K}}\label{sec:kadapt}

We consider in this section problems \Kadapt{K} and \finiteadapt. It is well-known that \Kadapt{K} mediates between \static and \twostage. Indeed, \Kadapt{1} is equivalent to \static, while, in case $\U$ (or $\Y$) is finite, \Kadapt{|\U|} (or \Kadapt{|\Y|}) is equivalent to \twostage. Hence, \twostage reduces to \finiteadapt in this case.  However, no rigorous analysis of the complexity landscape of the latter problems exists showing how their complexities are connected. When considering \Kadapt{K} as an approximation of \twostage an important unanswered question is: Is \Kadapt{K} actually easier to solve than \twostage and can it be harder than \static? Furthermore, in general it is not known whether \finiteadapt can be harder than \twostage. We will answer these questions in this section by providing general reductions between the problems. The most interesting results of this section concern reductions between \boundedKadapt{\k} and \boundedKadapt{(\k+1)}, even in the reverse direction in case of $\Udx$. To keep the presentation of this section as simple as possible, we focus on the bounded versions in what follows, although some of our membership results presented at the end of the section readily extend to the unbounded versions. Any proven hardness results are valid for the unbounded version as well.

A first observation is that, since the recourse variables can be merged into the first-stage variables when both are discrete, $\boundedKadapt{1}(\Xd,\Ud,\Yd)$ is equivalent to $\boundedstatic(\Xd,\Ud)$ and is therefore $\Sigma_2^p$-complete, while the equivalence with $\twostage(\Xd,\Ud,\Yd)$ implies that $\boundedKadapt{|\Ud|}(\Xd,\Ud,\Yd)$ is $\Sigma_3^p$-complete. Hence, with increasing $K$ from polynomial functions to exponential functions of the input size, the problem jumps one level up in the polynomial hierarchy; in particular, in the latter case the predicate 
$$
\{y^1,\ldots,y^\k\}\cap\Y(x,u)\neq\emptyset
$$
cannot, in general, be computed in polynomial time. 

As a main result of this section we show that for certain choices of $(\X,\U,\Y)$ the following reduction chain holds 
\begin{equation}\label{eq:chain_reduction_kadapt}
\begin{aligned}
\boundedstatic & \le \boundedKadapt{K} \\
& \le \boundedtwostage \\
& \le \boundedfiniteadapt .
\end{aligned}
\end{equation}

We start with showing the monotonicity of the hardness of $\boundedKadapt{\k}(\X,\U,\Y)$ with respect to increasing $\k$, as long as it remains polynomial.

\begin{lemma}\label{lem:k_to_k+1}
Let $\X\in \{ \X_c,\X_d\}$, $\U\in \{\U_c,\U_c(x),\U_d,\U_d(x)\}$ and $\Y\in \{\Y_c,\Y_d\}$. Then for any $\k\in \mathbb N$ it holds
\[
\boundedKadapt{K}(\X,\U,\Y) \le \boundedKadapt{(K+1)}(\X,\U,\Y).
\]
\end{lemma} 
\begin{proof}
Assume we are given an instance $I=\langle \X,\U,\Y\rangle$ to $\boundedKadapt{\k}(\X,\U,\Y)$.
We create an instance $I'=\langle \X',\U',\Y'\rangle$ of $\boundedKadapt{(K+1)}(\X,\U,\Y)$ as follows: We define $\X'=\X$ and the uncertainty set $\U'=\U\times \{ 0,1 \}$ in the discrete case and $\U'=\U\times [ 0,1 ]$ in the continuous case. For the second-stage feasibility set we define for $u\in \U'$
\[
\Y'(x,u)\equiv\set{(y,y_{n_y+1})\in \mathbb R_{\ge 0}^{n_y+1}}{T(u_{[n_u]})x + Wy - My_{n_y+1}\one \leq 0 , \ y_{n_y+1} \le 2u_{n_u+1}, \ y_{n_y+1} \ge 2u_{n_u+1} -1 },
\]
where
$u_{[n_u]}$ denotes the vector of the first $n_u$ components of $u$, and $$M=\max_{\substack{u\in\{ 0, R\}^{n_u} \\ x\in \{ 0,R\}^{n_x} \\ y\in\{0,R\}^{n_y}}} \max_{i\in[m_y]} (T(u)x +Wy)_i,$$ which is polynomial in the input and can be calculated in polynomial time by analyzing the coefficients appearing in $T(u)$ and $W$. In the discrete case we intersect $\Y'(x,u)$ with $\mathbb Z^{n_y+1}$. First note that $I'$ is indeed an instance of $\boundedKadapt{(K+1)}(\X,\U,\Y)$. We show that $I$ is a YES-instance if and only if $I'$ is a YES-instance. 

First assume $I$ is a YES-instance, i.e., there exist $x\in \X$ and $y^1,\ldots ,y^K\in \mathbb R_{\ge 0}^{n_y}$ (or $\mathbb Z_{\ge 0}^{n_y}$ in the discrete case) such that for every $u\in \U$ there exists a $k\in [K]$ such that $y^k\in \Y(x,u)$. We construct the following solution for the instance of $I'$: $\hat{x} = x$, $\hat{y}^1=(y^1,0),\ldots ,\hat{y}^K = (y^K,0), \hat{y}^{K+1} = (\zero,1)$. We have to show that for any $u'\in \U'$ one of the latter $K+1$ solutions lies in $\Y'(x,u')$. Consider an arbitrary $u'\in \U'$. If $u'_{n_u+1}\ge \frac{1}{2}$ then $\hat{y}^{K+1}$ is feasible for constraints $y_{n_y+1} \ge 2u_{n_u+1} -1$ and $y_{n_y+1} \le 2u_{n_u+1}$. Due to the definition of $M$ it is also feasible for all other constraints. On the other hand, if $u'_{n_u+1}<\frac{1}{2}$ then all of the solutions $\hat{y}^1,\ldots ,\hat{y}^K$ are feasible for the constraints $y_{n_y+1} \ge 2u_{n_u+1} -1$ and $y_{n_y+1} \le 2u_{n_u+1}$. Furthermore, the rest of the constraints are equivalent to the original constraints and since $y^1,\ldots ,y^K$ are feasible for the original instance at least one of them must be feasible for every $u_{[n_u]}\in \U$. 

For the other implication assume now that the constructed instance $I'$ is a YES-instance, i.e., there exist $\hat{x}\in \X'$ and $\hat{y}^1,\ldots , \hat{y}^{K+1}\in \mathbb R^{n_y+1}$ such that for every $u'\in \U'$ there exists a $k\in [K+1]$ such that $\hat{y}^k\in \Y'(\hat{x},u')$. We construct a solution to $I$ as follows: We set $x=\hat{x}$ and $y^k = (\hat y_1^k,\ldots ,\hat y_{n_y}^k)$ for all $k\in [K+1]$ for which $\hat y_{n+1}^k=0$ and for all other $k$ we use duplicates. Note that $\hat y_{n_y+1}^k=0$ holds for at most $K$ of the solutions since at least one of the solutions must be feasible for scenario $(u,1)$ and hence $\hat y_{n_y+1}^k\ge 1$ must hold for that solution due to constraint $y_{n_y+1} \ge 2u_{n_u+1} -1$. Then we know that $\hat{x}\in \X$ and for every $u\in \U$ we have that at least one solution $\hat{y}^k$ must be feasible for the scenario $(u,0)$, which can only hold if $\hat{y}^k_{n_y+1} = 0$. Hence, it must hold $T(u)x + Wy^k \leq 0$. It follows that $y^k \in \Y(x,u)$ and that $I$ must be a YES-instance. 
\end{proof}

It follows from the latter lemma that we can reduce \Kadapt{1} to \Kadapt{K}, if $K$ is fixed or polynomial in the input, by applying a polynomial sequence of reductions
\[
\boundedKadapt{1}(\X,\U,\Y)\le \cdots \le \boundedKadapt{K}(\X,\U,\Y).
\]
We now prove that $\boundedKadapt{K}\le \boundedtwostage$ under slightly stricter assumptions.

\begin{lemma}\label{lem:k_to_2ro}
Let $\X\in \{ \X_c,\X_d\}$, $\U\in \{\U_c,\U_c(x),\U_d,\U_d(x)\}$ and $\Y = \Y_d$. Then for any $\k\in \mathbb N$ it holds
\[
\boundedKadapt{K}(\X,\U,\Y) \le \boundedtwostage(\X,\U,\Y).
\]
\end{lemma} 
\begin{proof}
Assume we are given an instance
$I=(\X,\U, \Y)$
to $\boundedKadapt{K}(\X,\U,\Y)$. We create an instance
$I'=(\X',\U', \Y')$
of $\boundedtwostage(\X,\U,\Y)$ as follows: The first-stage feasibility set is $(\hat{x},x^1,\ldots ,x^\k) \in \X' \equiv\X\times [0,R]^{n_y} \times \cdots \times [0,R]^{n_y}$ if $\X$ is continuous and $(\hat{x},x^1,\ldots ,x^\k) \in \X' \equiv\X\times [R]_0^{n_y} \times \cdots \times [R]_0^{n_y}$ in case $\X$ is integer. We define the uncertainty set $\U'=\U$ and for the second-stage feasibility set we define for $u\in \U'$
\begin{align*}
\Y'(x,u)\equiv\Bigg\{(\hat{y}^1,\ldots ,\hat{y}^\k,\hat z)\in [R]_0^{K\times n_y} \times \{ 0,1\}^\k \ | \ &T(u)x + W\hat{y}^k \le M(1-\hat z_k)\one,\; \forall k\in [\k], \\
& \hat{y}^k = x^k \ \forall k\in [\k], \ \sum_{k\in[\k]} \hat z_k\ge 1\Bigg\},
\end{align*}
where $x = (\hat{x},x^1,\ldots ,x^\k)$, and $$M=\max_{\substack{u\in\{ 0, R\}^{n_u} \\ x\in \{ 0,R\}^{n_x} \\ y\in\{ 0,R\}^{n_y}}} \max_{i\in[m_y]} (T(u)x+W y)_i,$$ which is polynomial in the input and can be calculated in polynomial time by analyzing the coefficients appearing in $T(u)$ and $W$. First note that $I'$ is indeed an instance of $\boundedtwostage(\X,\U,\Y)$. We show that $I$ is a YES-instance if and only if $I'$ is a YES-instance. 

First assume $I$ is a YES-instance, i.e., there exist $x\in \X$ and $y^1,\ldots ,y^K\in \Y$ such that for every $u\in \U$ there exists $k\in [\k]$ such that $y^k\in \Y(x,u)$. Consider the following solution for the instance $I'$: $\hat{x} = x$, $x^1=y^1,\ldots ,x^\k = y^\k$. We have to show that for any $u'\in \U'$ there exists a feasible solution in $\Y'((\hat{x},x^1,\ldots ,x^\k),u')$. Consider an arbitrary $u'\in \U'$. Then we know that for one of the solutions $y^{k_0}$ it must hold $ T(u)x + Wy^{k_0} \le 0$ since $I$ is a YES-instance. Hence, the solution $\hat{y}^{k} = x^k= y^{k}$ for each $k\in[K]$ and $\hat z_{k_0} = 1$ is a feasible second-stage solution for $I'$ due to the definition of the big-M value.

For the other implication assume now that the constructed instance $I'$ is a YES-instance, i.e., there exist $(\hat{x},x^1,\ldots ,x^\k)\in \X'$ such that for every $u'\in \U'$ there exists a feasible second-stage solution $(\hat{y}^1,\ldots ,\hat{y}^\k,\hat z)\in \Y'(x,u)$. We construct a solution to $I$ as follows: We set $x=\hat{x}$ and $y^k = x^k$ for all $k\in [\k]$. We have to show that for every $u\in \U$ there exists a $y^k$ such that $T(u)x + Wy^k \le 0$. Since $I'$ is a YES-instance there must exist a feasible second-stage solution $(\hat{y}^1,\ldots ,\hat{y}^\k,\hat z)$ and hence an index $k_0$ with $\hat z_{k_0}=1$. It follows from the other constraints that $T(u)x + W\hat{y}^{k_0} \le 0$ which proves the result.
\end{proof}
We now consider \finiteadapt. Since $\k$ is part of the input, this class contains the bounded two-stage robust problem if either $\U$ or $\Y$ is a finite set. 
\begin{lemma}\label{lem:finite_adapt_harder_2ro}
If either $\U=\U_d$, $\U=\U_d(x)$ or $\Y=\Y_d$ then $$\boundedtwostage(\X,\U,\Y)\le \boundedfiniteadapt(\X,\U,\Y).$$ 
\end{lemma}
\begin{proof}
For any given instance of $\boundedtwostage(\X,\U,\Y)$, we define a corresponding instance of \linebreak$\boundedfiniteadapt(\X,\U,\Y)$ where all sets are exactly the same and $\k \ge \min\{|\U|,|\Y|\}$ where $|\Y|=\max_{x\in \X, u\in \U} |\Y(x,u)|$. Note that such a value for $\k$ can be obtained by setting $\k=(R+1)^{n_u}$ if $\U$ is discrete or $\k=(R+1)^{n_y}$ if $\Y$ is discrete which has encoding length polynomial in the input. Then both instances are equivalent since the maximum number of distinct second-stage policies needed in the inner exist-quantifier for $\boundedtwostage(\X,\U,\Y)$ is either one per scenario in $\U$ or the number of second-stage solutions $|\Y|$. Hence an instance of $\boundedfiniteadapt(\X,\U,\Y)$ is a YES-instance if and only if it is a YES-instance of $\boundedtwostage(\X,\U,\Y)$.
\end{proof}
The latter three lemmata lead directly to the following theorem.
\begin{theorem}\label{thm:k_k+1_reduction}
Let $\X\in \{ \X_c,\X_d\}$, $\U\in \{\U_c, \Uc(x), \U_d, \Ud(x)\}$ and $\Y\in \{\Y_c,\Y_d\}$. If $\k$ is polynomial in the input, then $$\boundedKadapt{1}(\X,\U,\Y) \le \boundedKadapt{K}(\X,\U,\Y).$$ 
Furthermore, if $\Y = \Y_d$, then 
$$\boundedKadapt{K}(\X,\U,\Y) \le \boundedtwostage(\X,\U,\Y).$$
Furthermore, if either $\U=\U_d$, $\U=\Ud(x)$ or $\Y=\Y_d$ then $$\boundedtwostage(\X,\U,\Y)\le \boundedfiniteadapt(\X,\U,\Y).$$ 
\end{theorem}
\begin{proof}
To prove the first result, we use Lemma \ref{lem:k_to_k+1} to create a sequence of reductions
\[
\boundedKadapt{1}(\X,\U,\Y)\le \cdots \le \boundedKadapt{K}(\X,\U,\Y)
\]
which is polynomial in the input if $\k$ is independent of or polynomial in the input. The second and third results follow directly from Lemma \ref{lem:k_to_2ro} and Lemma \ref{lem:finite_adapt_harder_2ro}.
\end{proof}

The latter result shows a hardness hierarchy for all four problems.

\begin{corollary}
If $\k$ is polynomial in the input, then $\boundedKadapt{K}(\X,\U,\Y)$ is at least as hard as $\boundedstatic(\X,\U)$, and, if $\Y=\Yd$ and $\U\in\{\Uc,\Ud\}$, also at least as hard as $\boundedstatic(\Xd,\U)$. Furthermore, $\boundedtwostage(\X,\U,\Y)$ is at least as hard as $\boundedKadapt{K}(\X,\U,\Y)$ if $\Y=\Y_d$, and $\boundedfiniteadapt(\X,\U,\Y)$ is at least as hard as $\boundedtwostage(\X,\U,\Y)$ if either $\U=\U_d$, $\U=\Ud(x)$ or $\Y=\Y_d$.
\end{corollary}
\begin{proof}
Any instance of $\boundedstatic(\X,\U)$ is an instance of $\boundedKadapt{1}(\X,\U,\Y)$ with $\Y(x,u)=\{\zero\}$ and the uncertain constraints written as $T(u)x\leq 0$. If $\Y=\Yd$ and $\U\in\{\Uc,\Ud\}$, an instance of $\boundedstatic(\Xd,\U)$ is obtained by relaxing $\Xd$ to $\Xc$ if necessary and adding integer recourse variables $z$ with $z=x$ to $\Y(x,u)$, as in Observation~\ref{obs:calld_hard}. The result then follows from Theorem~\ref{thm:k_k+1_reduction}.
\end{proof}

The latter corollary shows that, if $\Y=\Yd$, the complexity of $\boundedKadapt{K}(\X,\U,\Y)$ lies somewhere ``between'' static and two-stage robust optimization. One consequential question is: can $\boundedKadapt{K}(\X,\U,\Y)$ actually be harder than the corresponding static problem? The hardness of $\boundedKadapt{2}(\Xc,\Uc,\Yc)$ shown in \cite{bertsimas2010finite}, while $\boundedstatic(\Xc,\Uc)$ is in $\P$, illustrates that the answer can be yes in some settings (unless $\P=\NP$). Somewhat surprisingly however, we show in the following that for decision-dependent uncertainty and if all sets are integer, $\boundedstatic(\Xd,\Udx)$ is as hard as $\boundedKadapt{K}(\Xd,\Udx,\Yd)$. Together with the previous results we can conclude that both problems are equivalent in terms of hardness.
\begin{lemma}
If $K$ is polynomial in the input, then $$\boundedKadapt{(K+1)}(\Xd,\Udx,\Yd) \le \boundedKadapt{K}(\Xd,\Udx,\Yd).$$ 
\end{lemma} 
\begin{proof}
Assume we are given an instance
$I=(\X,\U, \Y)$
of $\boundedKadapt{(K+1)}(\Xd,\Udx,\Yd)$. After scaling we may assume that all data $T(u),W$ defining $\Y$ is integer. Together with the integrality of $x$, $u$ and $y$, this implies that $T(u)x+Wy$ is an integer vector for every feasible $x,u,y$. In particular, a policy $y$ is \emph{infeasible} for a scenario $u$ given $x$ if and only if $(T(u)x+Wy)_i\ge 1$ for at least one row $i\in[m_y]$.
 
We construct an instance
$I'=(\X',\U', \Y')$
of $\boundedKadapt{K}(\Xd,\Udx,\Yd)$ as follows. The first stage is enlarged by a solution $\hat x$, which will play the role of the discarded $(K+1)$-th policy:
$$(x',\hat x) \in \X' \equiv\X\times [R]_0^{n_y}.$$
The decision-dependent uncertainty set is
\begin{align*}
\U'((x',\hat x)) = \Bigg\{ (u',q', z')\in \U(x') \times \{ 0,1\}^{m_y} \times \{ 0,1\}: \ & T(u')x' + W\hat x \le  Mq', \\ 
& T(u')x' + W\hat x \ge \one - M(\one-q'), \\
& z' \le \sum_{i\in [m_y]} q'_i \le m_yz'\Bigg\},
\end{align*}
where the existence of a polynomially bounded big-$M$ value $M$ follows from the bounded nature of the problem. The second-stage feasibility set for $(u',q',z')\in \U'(x)$ is defined as
\begin{align*}
\Y'((x',\hat x),(u',q',z'))\equiv\set{y'\in [R]_0^{n_y}}{T(u')x' + Wy' \le M(1-z')\one}.
\end{align*}
First note that $I'$ is indeed an instance of $\boundedKadapt{K}(\X,\U(x),\Y)$. We show that $I$ is a YES-instance if and only if $I'$ is a YES-instance. 

First assume $I$ is a YES-instance, i.e., there exist $x\in \X$ and $y^1,\ldots ,y^{K+1}\in \Y$ such that for every $u\in \U(x)$ there exists a $k\in [\k+1]$ such that $y^k\in \Y(x,u)$. Consider the following solution for the instance $I'$: $x' = x$, $\hat x = y^{\k+1}$ and $y'^k = y^k$ for all $k=1,\ldots ,\k$. We have to show that for any $(u',q',z')\in \U'((x',\hat x))$ one of the solutions $y'^k$ is feasible in $\Y'((x',\hat x),(u',q',z'))$. Consider an arbitrary $(u',q',z')\in \U'((x',\hat x))$. We distinguish two cases according to the value of $z'$.

If $z'=0$, then due to the big-M constraint in $\Y'$ any second-stage solution $y'\in [R]_0^{n_y}$ is feasible, and hence any of the solutions $y'^k$ is feasible.

If $z'=1$, then the last constraint in $\U'$ ensures that $\sum_{i\in[m_y]}q'_i\ge 1$, so there is a row $i_0$ with $q'_{i_0}=1$, and the corresponding constraint of $\U'$ gives $\left(T(u')x'+W\hat x\right)_{i_0}\ge 1$. Hence $\hat x=y^{K+1}$ is infeasible for the scenario $u'$, i.e.\ $y^{K+1}\notin\Y(x,u')$. Since $u'\in\U(x')=\U(x)$ and $I$ is a YES-instance, there is a $k\in[K+1]$ with $y^k\in\Y(x,u')$; as $y^{K+1}$ is infeasible, this index satisfies $k\le K$. Because $z'=1$, the constraints of $\Y'$ reduce to $T(u')x'+Wy'\le 0$, and $y'^k=y^k$ satisfies $T(u')x'+Wy'^k=T(u')x+Wy^k\le 0$. Thus, $y'^k\in\Y'((x',\hat x),(u',q',z'))$.

For the other implication, assume that $I'$ is a YES-instance, i.e.\ there exist $(x',\hat x)\in\X'$ and $y'^1,\ldots,y'^K\in[R]_0^{n_y}$ such that for every $(u',q',z')\in\U'((x',\hat x))$ there is a $k\in[K]$ with $y'^k\in\Y'((x',\hat x),(u',q',z'))$. We construct a solution to $I$ by setting $x=x'$, $y^k=y'^k$ for $k\in[K]$, and $y^{K+1}=\hat x$. Let $u\in\U(x)=\U(x')$ be arbitrary. Define $q'\in\{0,1\}^{m_y}$ by $q'_i=1$ if $\left(T(u)x'+W\hat x\right)_i\ge 1$ and $q'_i=0$ otherwise, and set $z'=1$ if $q'\neq\zero$ and $z'=0$ otherwise. Using the integrality of $T(u)x'+W\hat x$ and the choice of $M$, the triple $(u,q',z')$ satisfies all constraints of $\U'$. Hence $(u,q',z')\in\U'((x',\hat x))$.

If $\hat x$ is feasible for $u$, i.e.\ $T(u)x+W\hat x\le 0$, then $q'=\zero$ and $z'=0$, and $y^{K+1}=\hat x\in\Y(x,u)$, so the scenario is covered by $y^{K+1}$. Otherwise $z'=1$, and since $I'$ is a YES-instance there is a $k\in[K]$ with $y'^k\in\Y'((x',\hat x),(u,q',z'))$. As $z'=1$ this means $T(u)x+Wy'^k\le 0$, so $y^k=y'^k\in\Y(x,u)$ covers the scenario. In either case there is a $k\in[K+1]$ with $y^k\in\Y(x,u)$. As $u\in\U(x)$ was arbitrary, $I$ is a YES-instance.
\end{proof}

It follows from the latter lemma that we can reduce \Kadapt{\k} to \Kadapt{1}, if $\k$ is fixed or polynomial in the input, by applying a polynomial sequence of reductions
\[
\boundedKadapt{\k}(\Xd,\Udx,\Yd)\le \cdots \le \boundedKadapt{1}(\Xd,\Udx,\Yd).
\]
This shows the following theorem.
\begin{theorem}
If $\k$ is polynomial in the input, then $$\boundedKadapt{\k}(\Xd,\Udx,\Yd) \le \boundedKadapt{1}(\Xd,\Udx,\Yd).$$ 
\end{theorem}

\begin{corollary}
$\boundedstatic(\Xd,\Udx)$ and $\boundedKadapt{K}(\Xd,\Udx,\Yd)$ are equivalent in terms of hardness.
\end{corollary}

We now turn to discussing membership results for $\boundedKadapt{\k}$.

\begin{theorem}\label{thm:k-adapt_in_sigma2p}
    \label{thm:UdxSigma}
$\boundedKadapt{K}(\X,\Ud(x),\Y)$ is in $\Sigma_2^p$ for $\X\in \{ \X_c,\X_d\}$ and $\Y\in \{\Y_c,\Y_d\}$ if $K$ is polynomial in the input size.
\end{theorem}
\begin{proof}
Assume $I$ is a YES-instance to $\boundedKadapt{K}(\X,\Ud,\Y)$ and let $\hat x, \hat y^1,\ldots,\hat y^\k$ be an outer certificate. Then the outer certificate induces a cover of the finite scenario set $\Ud = \bigcup_{k=1}^{\k} \U_k$ where
\[
\U_k = \set{ u\in \Ud}{\hat y^k\in\Y(\hat x,u) }.
\]
Hence, every point in the set
\[
\S\equiv\set{(x,y)\in\X\times\R^{\nyt\times \k}}{y^k\in\Y(x,u),\; \forall k\in[\k],u\in\U_k}
\]
is an outer certificate. Since $\Ud$ is bounded and discrete every $u$ has polynomial encoding length and hence $\S$ is described by a finite set of constraints of  polynomial encoding length. Hence, if either $x$ or $y$ is continuous and since $\S\neq \emptyset$ there exists a point in $\S$ with polynomial encoding length which proves the result. 

For the decision-dependent uncertainty set $\Ud(x)$ changing the $x$-certificate might change the scenario set as well. In this case we can conclude that every point in
\[
\mathcal Q\equiv\set{(x,y)\in \S}{\Ud(x)=\Ud(\hat x)}
\] 
is an outer certificate of our instance. In case $\X=\Xd$ we know that $\hat x$ has polynomial encoding length and hence any scenario in $\Ud(\hat x)$ has polynomial encoding length and we can directly apply the proof of the decision-independent case above. Otherwise, and if $\Y=\Yd$ we know that $\hat y$ has polynomial encoding length and we can apply Lemma \ref{lem:Udx} to the $x$-space and conclude that there exists a point in $\mathcal Q$ which has polynomial encoding length. Finally, if $\X = \Xc$ and $\Y=\Yc$ we first consider the projection of $\mathcal S$ to the $x$-space which can be described by a finite set of inequalities of polynomial encoding length (for instance, using Farkas' lemma, considering only the extreme rays of the resulting cone and relying on Fact~\ref{fact:inthull}). Applying Lemma \ref{lem:Udx} to this set shows that a polynomially encodable $\tilde x$ exists which is contained in the projection and for which $\Ud(\tilde x)=\Ud(\hat x)$. Then again since $\mathcal S$ is a rational polyhedron of polynomial encoding length there must exist a $y\in \set{y\in\mathbb R^{n_y\times\k}}{(\tilde x,y)\in \mathcal S}$ of polynomial encoding length which proves the result.
\end{proof}

Next we show that for discrete first-stage variables and continuous, possibly decision-dependent, uncertainty sets the $\k$-adaptability problem is in $\NP$.
\begin{theorem}
    \label{thm:kadapt_inNP}
$\boundedKadapt{K}(\Xd,\Ucx,\Y)$ is in $\NP$ for $\Y\in \{\Y_c,\Y_d\}$, for fixed $K$.
\end{theorem}
\begin{proof}
Since any $\hat x\in\Xd$ is integral and bounded by $R$, it has polynomial encoding length, and $\U(\hat x)$ is a non-empty polytope described by inequalities of polynomial encoding length (in the decision-independent case, $B(\hat x)=B$).

We start with $\Y=\Yd$. Assume $I$ is a YES-instance to $\boundedKadapt{K}(\Xd,\Ucx,\Yd)$ and $(\hat x,\hat y^1,\ldots,\hat y^K)$ an outer certificate, which has polynomial encoding length as $\Xd$ and $\Yd$ are discrete and bounded. The outer certificate fails to be a YES-certificate iff there exists a mapping $\phi: [\k]\to [\my]$ and a $u\in \U(\hat x)$ such that $(T(u)\hat x+W\hat y^k)_{\phi(k)}> 0$ for all $k\in [\k]$. Hence, conversely to verify that it is a YES-certificate we have to show for every mapping $\phi: [\k]\to [\my]$ that the optimal value of the problem $$\max \set{\delta}{(\delta,u)\in \mathbb R\times \U(\hat x), \ (T(u)\hat x+W\hat y^k)_{\phi(k)}\ge\delta\ \forall k\in [\k]}$$ is non-positive. Since $\k$ is fixed the number of possible mappings is polynomial and for each we have to solve an LP described by inequalities of polynomial encoding length which can be done in polynomial time as well. 

For the case $\Y=\Yc$ the latter idea does not work, since $\hat y$ is not necessarily of polynomial encoding length. We may also assume, using the auxiliary-coordinate convention of the paper, that each row of the recourse system is of the form $h_i(x)^\top u+(Wy^k)_i\equiv(T(u) x+W y^k)_{i} \leq 0$, where the vector $h_i(x)$ collects the coefficients of the bilinear form $(T(u)x)_i$ and is a linear function of $x$. 
 
Let $I$ be a YES-instance with outer certificate $(\hat x,\hat y^1,\ldots,\hat y^K)$, with $\hat y^k\in[0,R]^\nyt$ for all $k$. As in the previous part of the proof the certificate can be verified to be of a YES-instance if for every mapping $\phi: [\k]\to [\my]$ we have
\[
 \max\set{\delta}{(\delta,u)\in \mathbb R\times \mathbb R_{\ge 0}^{n_u}, \ B(\hat x) u\leq b,\ \delta\leq h_{\phi(k)}(\hat x)^\top u+(W\hat y^k)_{\phi(k)}\ \forall k\in[\k]} \le 0.
\]
Since $\U(\hat x)$ is nonempty and bounded, the latter linear program is feasible and bounded, so by strong duality it has the same value as
\begin{equation}
\label{eq:dualLP}
\min\set{b^\top\mu+\sum_{k\in[\k]}\lambda_k\,(W\hat y^k)_{\phi(k)}}{(\lambda,\mu)\in D_{\phi,\hat x}},
\end{equation}
where
\[
D_{\phi,\hat x} \equiv \set{(\lambda,\mu)\in\R^{K}_{\ge 0}\times\R^{\muu}_{\ge 0}}{B(\hat x)^\top\mu\ge\sum_{k\in[K]}\lambda_k h_{\phi(k)}(\hat x),\ \textstyle\sum_k\lambda_k=1},
\]
and the minimum is attained. The dual feasible region $D_{\phi,\hat x}$ depends on $\hat x$ and the mapping $\phi$ but not on $\hat y$. Since $\hat x$ is integral and bounded by $R$, each $h_{\phi(k)}(\hat x)$ and $B(\hat x)$ have polynomial encoding length, so $D_{\phi,\hat x}$ is a rational polyhedron of polynomial encoding length, and the minimum in~\eqref{eq:dualLP} is attained at a vertex $(\lambda_{\phi,\hat x}^*,\mu_{\phi,\hat x}^*)$ of polynomial encoding length.
 
Hence, a certificate for the problem consists in $(\hat x, \lambda_{\phi,\hat x}^*,\mu_{\phi,\hat x}^*)$. To certify that $I$ is indeed a YES-instance we check $\hat x\in\Xd$, $(\lambda_{\phi,\hat x}^*,\mu_{\phi,\hat x}^*)\in D_{\phi,\hat x}$ for each $\phi: [\k]\to [\my]$, and that the linear system
\begin{equation}
\label{eq:recoveryLP}
\set{(y^1,\ldots,y^K)\in[0,R]^{K\nyt}}{\ b^\top\mu_{\phi,\hat x}^* + \sum_{k\in[K]}(\lambda^*_{\phi,\hat x})_k (Wy^k)_{\phi(k)}\le 0\; \forall \phi: [\k]\to [\my]}
\end{equation}
is feasible which can all be done in polynomial time because the number of mappings $\phi$ is polynomial due to fixed $\k$.
\end{proof}
Theorem~\ref{thm:kadapt_inNP} contrasts with the continuous case: $\boundedKadapt{K}(\Xc,\Ucx,\Y)$ is $\exists\R$-hard by Corollary~\ref{cor:XcUcx}, while for integral first-stage decisions the decision-dependence of $\Ucx$ disappears once $\hat x$ is fixed.

Finally, we show that for some setups the reductions shown in this section may be loose in the sense that the reduced problem can be located on a lower level. First, as mentioned before, for $K\geq 2$, $\boundedKadapt{K}(\Xc,\Uc,\Yc)$ is strictly harder than $\boundedstatic(\Xc,\Uc)$ (unless $\P=\NP$), since $\boundedstatic(\Xc,\Uc)$ is in $\P$, but $\boundedKadapt{K}(\Xc,\Uc,\Yc)$ is $\NP$-hard, as follows from \cite{bertsimas2010finite} and Theorem~\ref{thm:k_k+1_reduction}. Second, $\boundedtwostage(\Xd,\Ud,\Yd)$ is strictly harder than $\boundedKadapt{K}(\Xd,\Ud,\Yd)$ (unless the polynomial hierarchy collapses), since from Theorem~\ref{thm:k-adapt_in_sigma2p} it follows that $\boundedKadapt{K}(\Xd,\Ud,\Yd)$ is $\Sigma_2^p$-complete, while we show in Section \ref{sec:twostage} that $\boundedtwostage(\Xd,\Ud,\Yd)$ is $\Sigma_3^p$-complete. It remains an open question whether $\boundedfiniteadapt$ can be located at least a level higher than $\boundedtwostage$ for certain setups.

In Table \ref{tab:k-adapt}, we show an overview of all complexity results for $\boundedKadapt{\k}$. The hardness results follow directly from the hardness results for $\boundedstatic$ and Theorem \ref{thm:k_k+1_reduction}. Note that in case $(\Xc,\Yd)$ with $\U\in\{\Uc,\Ud\}$, we can reduce both $\boundedstatic(\Xd,\U)$ and $\boundedstatic(\Xc,\U)$ to $\boundedKadapt{\k}$ (imposing integrality of $x$ through recourse variables $z=x$), and hence use the harder of the two. 

We conclude the section by listing the open questions for \boundedKadapt{\k} and \boundedfiniteadapt.

\begin{question}
$\boundedKadapt{K}(\Xc,\Ud,\Yc)$ is $\coNP$-hard but we can only prove membership in $\Sigma_2^p$. Is it also $\Sigma_2^p$-hard?
\end{question}

\begin{question}
For each $\Y\in\{\Yc,\Yd\}$, what is the membership of $\boundedKadapt{K}(\Xc,\Uc(x),\Y)$? Is it in $\exists\R$?
\end{question}

\begin{question}
Is $\boundedKadapt{K}(\Xc,\Uc,\Y)$ in $\NP$ for $\Y\in\{\Yc,\Yd\}$?
\end{question}

\begin{question}
Is $\finiteadapt(\Xc,\Uc,\Yc)$ or $\finiteadapt(\Xc,\Ud,\Yc)$ on a higher level of the polynomial hierarchy than \boundedtwostage?
\end{question}

\section{Conclusion}

This paper has been devoted to studying the complexity of robust problems of the form \static, \twostage, and \boundedKadapt{\k}. Our results summarized in Tables \ref{tab:static}--\ref{tab:k-adapt} show a diverse landscape of complexity results, however certain high-level tendencies can be extracted. First, we show that, in general, \static is easier than \twostage, while, for discrete recourse, \boundedKadapt{\k} is located ``between'' both problems in terms of hardness, sometimes strictly harder than \static, sometimes strictly easier than \twostage. With continuous recourse, \boundedKadapt{\k} can even be harder than \boundedtwostage (unless $\NP=\coNP$): $\boundedKadapt{K}(\Xc,\Uc,\Yc)$ is $\NP$-hard for $K\geq 2$, while $\boundedtwostage(\Xc,\Uc,\Yc)$ is $\coNP$-complete. Interestingly, in the discrete case with decision-dependent uncertainty, \static and \boundedKadapt{\k} are equivalent in terms of hardness. When \boundedKadapt{\k} is used as an approximation for \twostage, the latter results can provide guidance on the question of whether this approximation is actually easier to solve than the original problem.

For \static, we can see that, when switching from continuous to discrete decisions, the problem moves up one level in the polynomial hierarchy for classical uncertainty sets, while this is not the case for decision-dependent uncertainty in general. The same effect can be observed when switching from continuous uncertainty to discrete uncertainty, where the problem moves up one level in the polynomial hierarchy for decision-independent uncertainty, as well as for decision-dependent uncertainty with bounded discrete decisions. With bounded continuous decisions and decision-dependent uncertainty, the problem instead moves from $\exists\R$-complete to $\Sigma_2^p$-complete, two classes whose relationship is unknown. In contrast, switching from the bounded to the unbounded version does not increase the hardness of the problem in case of classical uncertainty, while, in case of decision-dependent uncertainty, this can make the problem undecidable. In particular, the combination of discrete decisions, decision-dependent uncertainty, and unboundedness makes the problem undecidable. An interesting observation is that, for continuous decision-dependent uncertainty, in the bounded case, the problem is harder for continuous decisions than for discrete decisions.

For \twostage, when switching from continuous to discrete first-stage decisions, the problem moves up in the polynomial hierarchy for continuous second-stage decisions, while the same clear effect cannot be seen for discrete second-stage decisions. Interestingly, our results, in general, do not show a change in complexity when switching from continuous to discrete uncertainty, which is in contrast to the static case. Switching from continuous to discrete second-stage decisions increases the level in the polynomial hierarchy most of the time. Similarly to the static case, switching from the bounded to the unbounded case makes the problem harder only in case of decision-dependent uncertainty and if either the first-stage or the second-stage decisions are discrete. Finally, it can be observed that the combination of continuous first-stage decisions and continuous decision-dependent uncertainty moves the problem out of the polynomial hierarchy towards $\exists\R$.

For \static and \twostage, while most of the questions have been answered, some membership questions remain open. Specifically, the only open question for \static concerns the case of $\static(\Xc,\Udx)$ for which we are unable to construct small certificates for YES-instances; the problem might even be undecidable. \twostage presents more open questions:
For all pairs $\X\in\{\Xc,\Xd\}$ and $\U\in\{\Uc,\Ud\}$,
it is open whether $\twostage(\X,\U,\Yd)\in\Sigma_3^p$.
Moreover, it remains open whether $\boundedtwostage(\Xc,\Ucx,\Yc)\in\Sigma_2^p$ (while hardness for $\Sigma_2^p$ and membership for $\Sigma_2\R$ are known), whether $\twostage(\Xc,\Udx,\Yc)\in\Sigma_2^p$ (while hardness for $\Sigma_2^p$ is known, but no membership; the problem might be undecidable), and whether $\boundedtwostage(\Xc,\Ucx,\Yd)\in\Sigma_3^p$ (while hardness for $\Sigma_3^p$ and membership for $\Sigma_3\R$ are known).

For $\boundedKadapt{\k}$, as expected, the overall picture shows that the problem is at least as hard as \boundedstatic and that the hardness is non-decreasing in the parameter $\k$. In particular, if $\Y=\Yd$, \boundedtwostage is at least as hard as $\boundedKadapt{\k}$, and if $\U\in\{\Ud,\Udx\}$ or $\Y=\Yd$, \boundedfiniteadapt is at least as hard as \boundedtwostage. 
One of the most important open cases for $\boundedKadapt{\k}$ is the one of $\boundedKadapt{K}(\Xc,\Ud,\Yc)$, for which we show $\coNP$-hardness, but we can only prove membership in $\Sigma_2^p$. The problem might be $\Sigma_2^p$-hard. Another interesting open question is whether there exists a setup where $\boundedfiniteadapt$ is at least one level higher in the polynomial hierarchy than \boundedtwostage.

\section*{Disclosure of generative AI tool usage}
The authors have made use of Anthropic’s Claude for proofreading the text, including grammar and stylistic improvements, and for helping to correct small mathematical mistakes. In addition, Claude was used to complete missing parts of the proofs of Theorems~\ref{thm:XcUxYdundecidable},~\ref{thm:XcUcxYcH} and~\ref{thm:kadapt_inNP}.

\end{document}